\documentclass{amsart}
\usepackage{amscd,amsmath,amssymb,amsthm}
\usepackage{graphics,graphicx}
\usepackage{xcolor}
\usepackage{hyperref}
\hypersetup{pdfpagemode=FullScreen,  
colorlinks=true}

\newtheorem{theorem}{Theorem}[section]
\newtheorem{prop}[theorem]{Proposition}
\newtheorem{lemma}[theorem]{Lemma}
\newtheorem{cor}[theorem]{Corollary}

\theoremstyle{definition}
\newtheorem{defn}[theorem]{Definition}
\newtheorem{remark}[theorem]{Remark}
\newenvironment{rem}{\begin{remark}\rm}{\end{remark}}
\newtheorem{example}[theorem]{Example}
\newenvironment{ex}{\begin{example}\rm}{\end{example}}

\def\Z{{\mathbb{Z}}}
\def\R{{\mathbb{R}}}

\mathchardef\shy="2D

\def\fr{{\rm fr}}
\def\cl{{\rm cl}}
\def\v{{\rm v}}
\def\w{{\rm w}}
\allowdisplaybreaks

\begin{document}

\title[Singular Locus Stratification and Closed Geodesics]{The Stratification of Singular Locus and \\
Closed Geodesics on Orbifolds}
\author{George C. Dragomir}
\maketitle

\begin{abstract} In this note, we prove the existence of a closed geodesic of positive length on any compact developable orbifold of dimension 3, 5, or 7. The argument uses the stratification of the singular locus, and reduces the problem of existence of a closed geodesic on a compact developable orbifold to the case of even dimensional orbifolds with zero-dimensional singular locus and orbifold fundamental group infinite torsion and of odd exponent.\\
\end{abstract}

\section*{Introduction}

The study of geodesics on manifolds has been a central topic in global differential geometry since its early beginnings in the 19th century, and it continues to be an active area of investigation today, see \cite{Bur:13Ope:aa} for a recent survey. Indeed, one of the most beautiful problems in Riemannian geometry is the question of existence of closed geodesics, and this has lead to numerous discoveries, notably the development of the calculus of variations in the large  \cite{Mor:96The:aa}  and Morse theory  on manifolds \cite{Mil:63Mor:aa}. For a manifold $M$, the closed geodesics on it are precisely the critical points of the energy functional $E$ defined on the free loop space $\Lambda M$, and Morse theory provides information about the existence and the number of critical points of $E$ in terms of the topology of $\Lambda M$, which is determined entirely by the topology of $M.$

The existence of closed geodesics on closed non-simply connected manifolds was established before the full development of Morse theory. The connected components of the free loop space $\Lambda M$ are in one-to-one correspondence with the conjugacy classes in the fundamental group $\pi_1(M)$, and a classical theorem of Cartan \cite{Car:28Lec:aa} states that if $M$ is closed, then each component of $\Lambda M$ corresponding to a nontrivial homotopy class contains a closed geodesic. The idea of the proof is to shorten a free loop on $M$ within its homotopy class and use the compactness of $M$ to show that this shortening process converges to a loop of minimum length. Because the injectivity radius of a compact manifold is positive, any sufficiently small loop will be contained in an $n$-ball and will therefore be homotopically trivial. Thus if one starts with a loop representing a nontrivial class in $\pi_1(M)$, the shortening process produces a loop of positive length representing a nontrivial closed geodesic on $M.$

The existence of closed geodesics on closed simply connected manifolds is more delicate and here the history is more storied. In 1917, Birkhoff used the variational approach to formulate the ``minimax principle" which was a remarkable breakthrough  and gives existence of closed geodesics on general surfaces \cite{Bir:17Dyn:aa} and on higher dimensional Riemannian manifolds homeomorphic to $\mathbb{S}^n$ (see \cite[p.193]{Bir:27Dyn:aa}). The basic idea is to consider an infinite family of curves on the manifold, choose a representative of maximum length and show how the original family can be deformed in a continuous way so that the length of such representative curves will attain a minimum value. Inspired by Birkhoff's minimax principle, Morse sought to apply the calculus of variations to the problem of existence of geodesics and the result was Morse theory on manifolds \cite{Mor:96The:aa}.

The complete answer to the question of existence of closed geodesics on closed Riemannian manifolds came in 1951 with the celebrated theorem of Fet and Lusternik \cite{Lyu:51Var:aa}. The proof uses the path-loop fibration $\Lambda M \to M$ to show that when $M$ is compact, the homotopy groups of $\Lambda M$ cannot be all trivial, and they then use this fact to find a constraint against which to minimize the energy functional and use it to produce a critical point. The proof relies on Morse's principle of polygonal approximation of the loop space, and in \cite{Kli:78Lec:aa} Klingenberg gives a more general proof of the Fet-Lusternik theorem by applying Morse theory directly to the infinite-dimensional space $\Lambda M$.

In this paper, we will consider the problem of existence of closed geodesics on orbifolds $\mathcal{Q}$, which are topological spaces locally modeled on quotients of a smooth manifold by the action of a finite group. This problem was first considered by Guruprasad and Haefliger in \cite{Gur:06Clo:aa}, who proved existence of closed geodesics for non-developable orbifolds and also for compact developable orbifolds $\mathcal{Q}$ with $\pi_1^{orb}(\mathcal{Q})$ either finite or containing an element of infinite order (cf. \cite[Theorem 5.1.1]{Gur:06Clo:aa}). 

We developed methods to demonstrate existence of geodesics on compact developable Riemannian orbifolds in several new contexts. Recall that an orbifold $\mathcal{Q}$ is called developable if it can be realized as a quotient $\mathcal{Q} = M/\Gamma$ of a complete simply connected Riemannian manifold $M$ by the geometric action of a subgroup $\Gamma \subseteq {\rm Isom}(M)$. (Here, an action is called \emph{geometric} if it is proper, cocompact, and acts by isometries.) In this case $M$ is the universal covering of $\mathcal{Q}$ and $\Gamma \cong \pi_1^{orb}(\mathcal{Q})$ is the orbifold fundamental group. Using the classification of isometries of $M$ we distinguish two classes of elements in $\Gamma$, the hyperbolic ones, which act without fixed points; and the elliptic ones, which act with fixed points. In \cite[Proposition 2.16]{Ale:11On-:aa} and \cite[Theorem 2]{Dra:14Clo:aa} it has been shown that the presence of hyperbolic elements in $\Gamma$ gives closed geodesics of positive length on $\mathcal{Q}.$ Therefore, the problem can be reduced to the case when all the elements in the orbifold fundamental group $\Gamma$ are elliptic. Because elliptic isometries have finite order, and because there are finitely many orbit types, the exponent of $\Gamma$ must be finite. Thus, based on \cite{Gur:06Clo:aa}, the existence problem remains open for compact developable orbifolds whose orbifold fundamental group is a finitely presented Burnside group (finitely generated infinite group of finite exponent). Again, no example of such a group is known, and there is no clear conjecture whether or not such groups exist. In \cite{Dra:14Clo:aa} we showed that infinite torsion groups of finite exponent cannot be realized as the fundamental group of compact orbifolds satisfying certain geometric conditions and, as a consequence, we obtained the existence of closed geodesics of positive length on all compact orbifolds of nonpositive or nonnegative sectional curvature (see  \cite[Corollary 3]{Dra:14Clo:aa}).

Here, we present a completely different approach to the existence problem. Instead of making curvature assumptions on the orbifold $\mathcal{Q},$ we base our arguments on the stratification of the singular locus $\Sigma \subset Q$. Each connected component of this stratification has the structure of a manifold of dimension less than or equal to $\dim(\mathcal{Q}),$ and when $\mathcal{Q}$ is assumed to be Riemannian, we show  that these components are totally geodesic (cf. Propostion \ref{prop:strata}). In general, these totally geodesic manifolds need not be compact, even when the ambient orbifold $\mathcal{Q}$ is, and thus there is no guarantee for the existence of closed geodesics in a given component. However, an interesting feature of this stratification is that we can compactify certain connected components by adding singular points with larger isotropy. The resulting compact space $S$ carries a natural non-effective orbifold structure $\mathcal{S}$ and continues to be totally geodesic in $\mathcal{Q}.$ Moreover, as Theorem \ref{thm:closed_stratum} shows, these spaces with the reduced orbifold structure $\mathcal{S}_{\rm eff}$ are either manifolds, or orbifolds with only zero-dimensional singular locus. Thus, if one can prove the existence of closed geodesics of positive length for compact orbifolds with only zero-dimensional singular locus, then the existence would follow for all compact orbifolds from Theorem \ref{thm:closed_stratum}  by using the totally geodesic suborbifold $\mathcal{S}_{\rm eff}$ (see also Remark \ref{rem:reduction}). 

Reducing the study of existence of closed geodesics to compact orbifolds $\mathcal{Q}$ with only zero-dimensional singular locus presents certain advantages. Because the singular locus of $\mathcal{Q}$ consists of isolated points, each isotropy group admits a free orthogonal action on the sphere $\mathbb{S}^{n-1}$, where $n=\dim(\mathcal{Q}),$ and this imposes additional restrictions on these isotropy groups \cite{Wol:67Spa:aa}. In particular, if $\dim(\mathcal{Q})$ is odd, then all the isotropy groups are cyclic of order two, and Theorem \ref{thm:order2point} implies that $\mathcal{Q}$ admits a closed geodesic of positive length. This fact plays a central role in the proof of Theorem \ref{thm:odd_strata}, which is used to show that all compact orbifolds of dimension 3, 5 or 7 have closed geodesics of positive length (Corollary \ref{cor:dimension_3_5}).

The outline of the paper is as follows. In the first section we review some of the basics of orbifolds and in \S\ \ref{sec:stratification} we describe the natural stratification by orbit type of an orbifold. In \S\ \ref{sec:geodesics01}, we present sufficient conditions on the structure of the singular locus of a general compact orbifold that give existence of closed geodesics, and 
show that the existence problem can be reduced to orbifolds with zero-dimensional singular locus. The case of compact developable orbifolds is considered in \S\ \ref{sec:develop_geodesics}.

 \section{Preliminaries}\label{sec:prelim} 

In this section we collect some basic background on orbifolds and fix the notation. For a comprehensive treatment of this material, the reader is encouraged to refer to the excellent introductions to orbifolds in \cite[Chapter 1]{Ade:07Orb:aa} and \cite[Chapter III.$\mathcal{G}$]{Bri:99Met:aa}, which also include the modern approach to orbifolds, using groupoids. While, for the most part, our presentation follows that of \cite{Ade:07Orb:aa} and \cite{Bri:99Met:aa}, because of our interest in the geometric properties of (developable) orbifolds, we choose to use the traditional approach to orbifolds, in the spirit of \cite{Sat:57The:aa} and \cite{Thu:78The:aa}. 

\subsection{Orbifolds}  

Let $Q$ denote a paracompact Hausdorff topological space and let $\mathcal{U}=\{U_i\}_{i\in\mathcal{I}}$ be a locally finite open cover of $Q$. Fix a non-negative integer $n$. 

An $n$-dimensional \emph{orbifold chart} associated to an open set $U_i\in \mathcal{U}$ is given by a triple $(\widetilde{U}_i,\Gamma_i,\varphi_i)$, where $\widetilde{U}_i$ is a simply connected smooth $n$-manifold, $\Gamma_i\subseteq {\rm Diff}(\widetilde{U}_i)$ is a finite group acting smoothly on $\widetilde{U}_i$, and $\varphi_i\colon \widetilde{U}_i\to U_i$ is a continuous surjective map that induces a homeomorphism from $\widetilde{U}_i/\Gamma_i$ onto the open set $U_i$. 

An $n$-dimensional \emph{orbifold atlas} $\mathcal{A}$ associated to the open cover $\mathcal{U}$ of  $Q$ is a collection of $n$-dimensional orbifold charts $\{(\widetilde{U}_i,\Gamma_i,\varphi_i)\}_{i\in\mathcal{I}}$  which satisfy the following compatibility condition: for all $\tilde{x}_i\in \widetilde{U}_i$ and $\tilde{x}_j\in \widetilde{U}_j$ with $\varphi_i(\tilde{x}_i)=\varphi_j(\tilde{x}_j)$ there exists a diffeomorphism $h\colon \widetilde{U}\to \widetilde{V}$ defined on a connected open neighbourhood $\widetilde{U}$ of $\tilde{x}_i$ onto a neighbourhood $\widetilde{V}$ of $\tilde{x}_j$ such that $\varphi_j\circ h=\varphi_i|_{\widetilde{U}}$. Such a map $h$ is called a \emph{change of charts} and it is well-defined up to composition with an element of $\Gamma_j$ (see Remark \ref{rem:effective}). If $i=j$, then $h$ is the restriction $\gamma|_{\widetilde{U}}$ of an element $\gamma\in\Gamma_i$.

Two orbifold atlases $\mathcal{A}_1=\{(\widetilde{U}_i,\Gamma_i,\varphi_i)\}_{i\in \mathcal{I}_1}$ and $\mathcal{A}_2=\{(\widetilde{U}_i,\Gamma_i,\varphi_i)\}_{i\in \mathcal{I}_2}$ on $Q$ are \emph{equivalent} if their union $\{(\widetilde{U}_i,\Gamma_i,\varphi_i)\}_{i\in \mathcal{I}_1\cup\mathcal{I}_2}$ satisfies the above compatibility condition. 

\begin{defn}\label{def:orbifold}
An \emph{effective orbifold} of dimension $n$ is a paracompact Hausdorff topological space $Q$ together with an equivalence class $[\mathcal{A}]$ of $n$-dimensional effective orbifold atlases on it.
\end{defn}

The orbifold structure $(Q,\mathcal{A})$ on the topological space $Q$ defined by an orbifold atlas $\mathcal{A}$ will be denoted by the calligraphic $\mathcal{Q}$. The space $Q$ is called the underlying topological space of the orbifold $\mathcal{Q}$. An orbifold $\mathcal{Q}$ is said to be \emph{connected} if the underlying topological space $Q$ is connected.

\begin{rem}\label{rem:effective}  
The following technical result plays an important role in studying orbifolds (see \cite[Proposition A.1]{Moe:97Orb:aa}, or \cite[Proposition 1.11]{Dra:11Clo:aa} for an alternative proof). Suppose $M$ is a paracompact connected smooth manifold and $\Gamma$ is a finite subgroup of $\mbox{Diff}(M)$. Let $V\neq \varnothing$ be a connected open subset of $M$ and $f\colon V\rightarrow M$ be a diffeomorphism onto its image such that  $\varphi\circ f = \varphi|_{V}$, where $\varphi\colon M\to M/\Gamma$ is the natural projection. Then there exists a unique $\gamma\in\Gamma$ such that  $f=\gamma|_{V}$. 
\end{rem}

The assumption that in each orbifold chart $(\widetilde{U}_i,\Gamma_i,\varphi_i)$ the group $\Gamma_i$ is a finite group of diffeomorphisms of the manifold $\widetilde{U}_i$ implies that  $\Gamma_i$ acts \emph{effectively} on $\widetilde{U}_i$, i.e. no element of the group $\Gamma_i$ besides the identity element fixes all the points in $\widetilde{U}_i$. This follows from the well known fact that the only diffeomorphism of finite order on a connected paracompact smooth manifold which fixes a nonempty open set is the identity map (see for example \cite[Corollary 1.9]{Dra:11Clo:aa}). Such orbifold charts are said to be \emph{reduced} and, an orbifold is \emph{effective} if it is given by an atlas consisting of reduced orbifold charts.

Non-effective orbifold structures appear naturally in the study of orbifolds and we will encounter examples of such structures in \S\ \ref{sec:stratification}, where we consider  orbifold structures of certain subsets contained entirely in the singular locus of the initial (effective) orbifold. 

To give a \emph{non-effective orbifold structure} $\mathcal{Q}$, in the definition of the orbifold charts $(\widetilde{U}_i,\Gamma_i,\varphi_i)$ assume instead that $\Gamma_i$ is a finite group acting smoothly by diffeomorphisms on the manifolds $\widetilde{U}_i$ such that the kernel $K_i=\ker\{\Gamma_i\to {\rm Diff}(\widetilde{U}_i)\}$  of the $\Gamma_i$-action is nontrivial.  A collection of such orbifold charts form an orbifold atlas if in addition to the local compatibility condition we require each change of charts $h\colon \widetilde{U}\to\widetilde{V}$ with $\widetilde{U}\subseteq\widetilde{U}_i$ and $\widetilde{V}\subseteq\widetilde{U}_j$ to induce an isomorphism between the kernels $K_i=\ker\{\Gamma_i\to {\rm Diff}(\widetilde{U}_i)\}$ and $K_j=\ker\{\Gamma_j\to {\rm Diff}(\widetilde{U}_j)\}$ of the corresponding actions.  Note that given a non-effective orbifold $\mathcal{Q}$ one can associate to it an effective one denoted $\mathcal{Q}_{\rm eff}$, by redefining the groups in each orbifold chart $(\widetilde{U}_i,\Gamma_i,\varphi_i)$ to be $\Gamma_i/K_i$. 

Finally, we note that if all the groups $\Gamma_i$ are trivial, or if they act freely on the manifolds $\widetilde{U}_i$, then $Q$ is a topological $n$-manifold and the orbifold atlas $\mathcal{A}$ defines a smooth $n$-manifold structure on $Q$. In the general case, however, the underlying topological space $Q$ of a smooth nontrivial orbifold $\mathcal{Q}$ need not have the structure of a topological manifold. 

\begin{defn}
An orbifold $\mathcal{Q}=(Q,\mathcal{A})$ is said to be \emph{Riemannian} if in each orbifold chart $(\widetilde{U}_i,\Gamma_i,\varphi_i)\in \mathcal{A}$, the uniformizing sets $\widetilde{U}_i$ are Riemannian manifolds, the groups $\Gamma_i\subseteq {\rm Isom}(\widetilde{U}_i)$ act on $\widetilde{U}_i$ by isometries, and the change of charts of the orbifold atlas $\mathcal{A}$ are Riemannian isometries. 
\end{defn}
\begin{rem}\label{rem:riemann_orb} 
Any smooth orbifold $\mathcal{Q}=(Q,\mathcal{A})$ admits a Riemannian structure (see for instance \cite[Proposition 2.24]{Dra:11Clo:aa}).
\end{rem}

A \emph{smooth map} between two orbifolds $\mathcal{Q}=(Q,\mathcal{A})$ and $\mathcal{Q}'=(Q',\mathcal{A}')$ is a continuous map $f\colon Q\to Q'$ with the property that for any point $x\in Q$ there exist charts $(\widetilde{U},\Gamma,\varphi)$ around $x$ and $(\widetilde{U}',\Gamma',\varphi')$ around $f(x)$, such that $f$ maps $U=\varphi(\widetilde{U})$ into $U'=\varphi'(\widetilde{U}')$ and $f$ can be lifted to a smooth map $\tilde{f}\colon \widetilde{U}\to\widetilde{U}'$ satisfying $\varphi'\circ \tilde{f}=f\circ\varphi$. 

Two orbifolds $\mathcal{Q}$ and $\mathcal{Q}'$ are said to be \emph{diffeomorphic} if there are smooth orbifold maps $f\colon Q\to Q'$ and $f'\colon Q'\to Q$ such that $f\circ f'=1_{Q'}$ and $f'\circ f=1_Q$.

\subsection{Isotropy Groups}\label{ssec:isotropy_groups} 

Suppose $\mathcal{Q}$ is an effective orbifold and fix $x\in \mathcal{Q}$. Let $(\widetilde{U}_i,\Gamma_i,\varphi_i)$ be an orbifold chart at $x$.  Choose $\tilde{x}_i\in\widetilde{U}_i$ such that $\varphi_i(\tilde{x}_i)=x$, and denote by $\Gamma_{\tilde{x}_i}$ the isotropy group of $\tilde{x}_i$ in $\Gamma_i$. That is, $\Gamma_{\tilde{x}_i} = \{\gamma\in \Gamma_i\mid \gamma\tilde{x}_i=\tilde{x}_i\}$. 

The isomorphism class of the group $\Gamma_{\tilde{x}_i}$ is independent of the choice of uniformizing chart at $x$.  To see this, let $(\widetilde{U}_j,\Gamma_j,\varphi_j)$ be another orbifold chart at $x$ and let $\tilde{x}_j\in \widetilde{U}_j$ be such that $\varphi_j(\tilde{x}_j)=x$. By definition, there exists a change of charts $h\colon \widetilde{U}\to\widetilde{V}$ defined on an open connected neighbourhood $\widetilde{U}$ of $\tilde{x_i}$ onto the neighbourhood $\widetilde{V}=h(\widetilde{U})$ of $\tilde{x}_j$ such that $h(\tilde{x}_i)=\tilde{x}_j$. Without loss of generality we can assume that the set $\widetilde{U}$ is $\Gamma_{\tilde{x}_i}$-invariant. Since for each $\gamma\in\Gamma_{\tilde{x}_i}$, the composition $h\circ\gamma\circ h^{-1}\colon \tilde{V}\to \widetilde{V}$ is a diffeomorphism of $\widetilde{V}$ satisfying $\varphi_j\circ (h\circ \gamma\circ h^{-1}) = \varphi_j|_{\widetilde{V}}$, the result in Remark \ref{rem:effective} implies that there exists a unique element $\delta\in \Gamma_j$ such that $h\circ\gamma\circ h^{-1}=\delta|_{\widetilde{V}}$. Thus the change of charts $h$ induces an injective homomorphism $\theta\colon \Gamma_{\tilde{x}_i}\to \Gamma_j$ given by $\theta(\gamma)=\delta$. Further, we claim that the homomorphism $\theta$ gives an isomorphism between $\Gamma_{\tilde{x}_i}$ and $\Gamma_{\tilde{x}_j}$. 

We first show that if $\lambda\in \Gamma_j$ is such that $\lambda.\widetilde{V}\cap \widetilde{V}\neq\varnothing$, then $\lambda\in {\rm Im}(\theta)$. Indeed, let $\lambda\in\Gamma_j$ and let $\widetilde{V}'$ be an open connected set contained in $\lambda.\widetilde{V}\cap \widetilde{V}\neq\varnothing$. Let $\widetilde{U}' = h^{-1}|_{\widetilde{V}'}(\widetilde{V}')$ and denote $h'=h|_{\widetilde{U}'}$. Then $\widetilde{U}'$ is an open connected subset of $\widetilde{U}$ and the composition $h'^{-1}\circ\lambda\circ h' \colon \widetilde{U}'\to \widetilde{U}'$ is a diffeomorphism which satisfies $\varphi_i|_{\widetilde{U}}\circ(h'^{-1}\circ\lambda\circ h') = \varphi_i|_{\widetilde{U}'}$. By Remark \ref{rem:effective}, there is a unique element $\gamma\in \Gamma_{\tilde{x}_i}$ such that $h'^{-1}\circ\lambda\circ h' = \gamma|_{\widetilde{U}'}$. Thus $\lambda = h\circ \gamma\circ h^{-1}$ on $\widetilde{V}'$ and therefore on $\widetilde{V}$. This shows that $\lambda = \theta(\gamma)$, i.e. $\lambda\in {\rm Im}(\theta)$. Finally, by shrinking the domain $\widetilde{U}$ of $h$ if necessary we can assume that the set $\widetilde{V} = h(\widetilde{U})$ is such that $\{\lambda\in\Gamma_j\mid \lambda.\widetilde{V}\cap \widetilde{V}\neq\varnothing\}=\Gamma_{\tilde{x}_j}$. This implies that ${\rm Im}(\theta)=\Gamma_{\tilde{x}_j}$, and thus  the monomorphism $\theta\colon \Gamma_{\tilde{x}_i}\to\Gamma_j$ maps $\Gamma_{\tilde{x}_i}$ isomorphically into $\Gamma_{\tilde{x}_j}$. 

The above argument, with $i=j$, can be used to show that the isomorphism class of the group $\Gamma_{\tilde{x}_i}$ is also independent of the choice of the lift $\tilde{x}_i$ of $x$ in $\widetilde{U}_i$. If $\tilde{x}_i'$ is another lift of $x$ in $\widetilde{U}_i$, then $\tilde{x}_i' = \gamma\tilde{x}_i$ for some $\gamma\in\Gamma_i$ and $\Gamma_{\tilde{x}_i'}=\gamma\Gamma_{\tilde{x}_i}\gamma^{-1}.$ 

In conclusion, given a point $x\in \mathcal{Q}$, the isomorphism class of the isotropy group $\Gamma_{\tilde{x}_i}$ is independent of both the choice of orbifold chart $(\widetilde{U}_i,\Gamma_i,\varphi_i)$ at $x$ and its lift $\varphi_i^{-1}(x)$ within the chosen orbifold chart. This isomorphism class is denoted by $\Gamma_x $ and is referred to as the \emph{isotropy group of $x$}. 

A point $x\in Q$ is said to be a \emph{singular point} if its isotropy group $\Gamma_x$ is nontrivial. A nonsingular point is also called a \emph{regular point}. We denote by $\Sigma$ the collection of all singular points in $\mathcal{Q}$, and by $Q_{\rm reg}$ the set of the regular ones. The singular locus $\Sigma$ is closed and has empty interior in $Q$ (cf. \cite[Proposition 1.10]{Dra:11Clo:aa}). If $\Sigma=\varnothing$ then the orbifold $\mathcal{Q}$ is in fact a smooth manifold. 

\begin{rem}\label{rem:fundam_nbhd} Each point $x\in \mathcal{Q}$ has an open neighbourhood $U_x$ (called \emph{fundamental neighbourhood at $x$}) such that the group of the associated orbifold chart is isomorphic to the isotropy group $\Gamma_x$ of $x$. We call such  chart a \emph{fundamental chart at $x$} and denote it by $(\widetilde{U}_x,\Gamma_x,\varphi_x)$.

From the compatibility condition of charts it follows that if $U_x$ is a fundamental neighbourhood at $x$, then the isotropy group of any point $y\in U_x$ is isomorphic to a subgroup of the isotropy group $\Gamma_x$ of $x.$ Moreover, a fundamental chart $(\widetilde{U}_y,\Gamma_y,\varphi_y)$ at $y$ can be chosen such that $\widetilde{U}_y\subset \widetilde{U}_x$, $\Gamma_y\le\Gamma_x$ and $\varphi_y=\varphi_x|_{\widetilde{U}_y}$. In particular, any point contained in a fundamental neighbourhood of a regular point is again regular.

An orbifold atlas $\mathcal{A} = \{(\widetilde{U}_i,\Gamma_i, \varphi_i)\}_{i\in\mathcal{I}}$ defining the orbifold structure $\mathcal{Q}$ can be refined to an equivalent orbifold atlas containing only fundamental charts.
\end{rem}

\subsection{Developable Orbifolds}\label{ssec:good_orbifolds} 

Let $M$ be a smooth $n$-manifold and let $\Gamma\subseteq{\rm {Diff}}(M)$ be a discrete group of diffeomorphisms. The action of $\Gamma$ on $M$ is said to be \emph{proper} (or $\Gamma$ acts \emph{properly} on $M$) if given any two compact sets $C$ and $C'$ in $M$, the set $\{\gamma\in \Gamma\mid \gamma.C\cap C'\neq\varnothing\}$ is finite. Equivalently (cf. \cite[Remark I.8.3(1) and I.8.4(2)]{Bri:99Met:aa} or \cite[Proposition 1.1]{Dra:11Clo:aa}), the action of $\Gamma$ on $M$ is proper if  and only if:
\begin{itemize}
\item[$(i)$] the space of orbits $M/\Gamma$ with the quotient topology is Hausdorff;
\item[$(ii)$] each $\tilde{x}\in M$ has a finite isotropy group $\Gamma_{\tilde{x}}$;
\item[$(iii)$] each $\tilde{x}\in M$ has a $\Gamma_{\tilde{x}}$-invariant open neighbourhood $\widetilde{U}_{\tilde{x}}$ such that $\{\gamma\in \Gamma\mid \gamma.\widetilde{U}_{\tilde{x}}\cap\widetilde{U}_{\tilde{x}}\neq\varnothing\}=\Gamma_{\tilde{x}}$. 
\end{itemize}

Denote by $Q=M/\Gamma$ the orbit space of the $\Gamma$-action on $M$ and let $\pi\colon M\to Q$ denote the quotient map. The sets $\pi(\widetilde{U}_{\tilde{x}})=\widetilde{U}_{\tilde{x}}/\Gamma_{\tilde{x}}$, with $\widetilde{U}_{\tilde{x}}$ given by $(iii)$ above, define an open cover of the space $Q$ and the collection $\{(\widetilde{U}_{\tilde{x}},\Gamma_{\tilde{x}},\pi|_{\widetilde{U}_{\tilde{x}}})\}_{\tilde{x}\in M}$ forms an $n$-orbifold atlas on $Q$. The change of charts of this orbifold atlas are restrictions of elements of $\Gamma$: if $h:\widetilde{U}\to \widetilde{V}$ is a change of charts, then there exists a unique $\gamma\in\Gamma$ such that $h=\gamma|_{\widetilde{U}}$. The orbifold structure given by this atlas depends only on the $\Gamma$ action on $M$ and not on the choice of the invariant open neighbourhoods  $\widetilde{U}_{\tilde{x}}$ {\cite[Proposition 13.2.1]{Thu:78The:aa}}. The orbifold $\mathcal{Q}=M/\Gamma$ is called \emph{the orbifold quotient of $M$ by the proper action of $\Gamma$}. 

An orbifold is called \emph{developable} (or \emph{good}) if it arises as the global quotient of a discrete group acting properly on a manifold. 

\subsection{Orbifold fundamental group}

A smooth map between two orbifolds $\mathcal{Q}'=(Q',\mathcal{A}')$ and $\mathcal{Q}=(Q,\mathcal{A})$ is called an \emph{orbifold covering map} if its underlying continuous map $p\colon  Q' \to Q$ is surjective and satisfies the following condition: for each point $x \in Q,$ there exists an open neighbourhood $U\subseteq Q$ of $x$ with orbifold chart $(\widetilde{U},\Gamma,\varphi)$ and a subgroup $\Gamma'\le\Gamma$ such that each connected component $U'$ of  $p^{-1}(U)$ in $Q'$ is uniformized by $(\widetilde{U},\Gamma',\varphi')$. In this case we say that  $\mathcal{Q}'$ is an \emph{orbifold covering} of the orbifold $\mathcal{Q}$ and, when there is no confusion for the orbifold structures involved, we say that $p\colon Q'\to Q$ is the corresponding orbifold covering map. However, it is important to note that, in general, the map $p\colon Q'\to Q$  is not a covering map of the underlying topological spaces. 

An orbifold covering map $p\colon \widetilde{Q}\to Q$ of a connected orbifold $\mathcal{Q}$ is called a \emph{universal covering map} if it satisfies the property that for any other orbifold covering map $p'\colon Q'\to Q$ there exists an orbifold covering $\overline{p}\colon \widetilde{Q}\to Q'$ such that $p=p'\circ\overline{p}$. If $p\colon \widetilde{Q}\to Q$ is a universal covering  map then the orbifold $\widetilde{\mathcal{Q}}$ is called the \emph{orbifold universal covering} of $\mathcal{Q}$. The orbifold $\widetilde{\mathcal{Q}}$ is well defined up to diffeomorphism of orbifolds.

Thurston showed that each orbifold $\mathcal{Q}$ has an orbifold universal covering \cite[Proposition 13.2.4]{Thu:78The:aa} and defined the \emph{orbifold fundamental group} $\pi_1^{orb}(\mathcal{Q})$ of $\mathcal{Q}$ to be the group of deck transformations of its orbifold universal covering. The isomorphism class of this group is an invariant of the orbifold structure $\mathcal{Q}$ and not just of the topology of the underlying space $Q$. In general, the orbifold fundamental group is not the fundamental group of the underlying topological space; and orbifolds with the same underlying topological space but different orbifold structures can have non-isomorphic orbifold fundamental groups. 

For a developable orbifold $\mathcal{Q}=M/\Gamma$, the quotient $M \to M/\Gamma$ can be regarded as an orbifold covering with $\Gamma$ as the group of deck transformations. Any subgroup $\Gamma'$  of $\Gamma$ induces an intermediate orbifold covering $M/\Gamma' \to M/\Gamma$, and any manifold covering $\widetilde{M} \to M$ gives an orbifold covering for $\mathcal{Q}$ by composing with the quotient map $M \to M/\Gamma$. In particular, the universal covering space of $M$ is the orbifold universal covering space of $\mathcal{Q}$, and the orbifold fundamental group belongs in a short exact sequence $$1 \to \pi_1 (M) \to \pi_1^{orb}(\mathcal{Q}) \to \Gamma \to 1.$$ 
Conversely, if the orbifold universal covering $\widetilde{\mathcal{Q}}$ of an orbifold $\mathcal{Q}$ is a manifold, then $\mathcal{Q}$ is the global quotient of $\widetilde{Q}$ by the proper action of $\pi_1^{orb}(\mathcal{Q})$. Therefore, an orbifold $\mathcal{Q}$ is developable if and only if its universal covering $\widetilde{\mathcal{Q}}$ is a manifold.

If $M$ is a Riemannian manifold and $\Gamma\subseteq{\rm Isom}(M)$ is a discrete group acting properly and by isometries on $M$, then the orbifold quotient $\mathcal{Q}=M/\Gamma$ carries a natural Riemannian orbifold structure induced by the one on $M$. Conversely, a Riemannian structure on a connected developable orbifold $\mathcal{Q}$, gives a natural Riemannian metric on the universal covering $M=\widetilde{\mathcal{Q}}$, the pull back by the covering map of the metric on $\mathcal{Q}$, and $\Gamma=\pi_1^{orb}(\mathcal{Q})$ acts as a group of isometries on $M$.  Moreover, in the light of Remark \ref{rem:riemann_orb}, any connected  developable orbifold $\mathcal{Q}$ can be obtained as the global quotient $M/\Gamma$ of a simply connected Riemannian manifold $M$ by the proper action of a discrete group of isometries $\Gamma\subseteq{\rm Isom}(M)$. The isomorphism class of the group $\Gamma$ is independent of the choice of Riemannian metric on $\mathcal{Q}$ and depends only on the orbifold structure $\mathcal{Q}$.

\subsection{Orbifold Geodesics}\label{ssec:paths}
 
A \emph{geodesic path} on an effective Riemannian orbifold $\mathcal{Q}=(Q,\mathcal{A})$ is a continuous path $c\colon I\to Q$ defined on an interval $I\subseteq\mathbb{R}$ and having the property that for any $t\in I$ there exists a subinterval $J\subset I$ containing $t$ and an orbifold chart $(\widetilde{U},\Gamma,\varphi)$ around $c(t)$ such that $c(J)\subset\varphi(\widetilde{U})$ and the restriction $c|_J$ lifts to a smooth geodesic $\tilde{c}\colon J\to\widetilde{U}$ satisfying $\varphi\circ\tilde{c}=c|_J$. 

The local lift of $c$ in an orbifold chart at $c(t)$ is not necessarily unique and giving an orbifold geodesic requires specifying a choice of such lift at each $t\in I$. When $I$ is a closed interval $[a,b]$, the image $c([a,b])\subset Q$ is compact and there exists a finite subdivision $a=t_0\le t_1\le \cdots\le t_k=b$ of the interval $[a,b]$ which is fine enough so that each restriction $c|_{[t_{i-1},t_i]}$ lifts to a smooth geodesic in some orbifold chart $(\widetilde{U}_i,\Gamma_i,\varphi_i)$ around $c([t_{i-1},t_i])$ for $i=1,\ldots,k$.  Over such a subdivision, a geodesic path over $c:[a,b]\to Q$ is given by a sequence $(\tilde{c}_1, h_1,\tilde{c}_2,\ldots, h_{k-1},\tilde{c}_k,h_k)$, where
\begin{itemize}
\item[$\cdot$] $\tilde{c}_i:[t_{i-1},t_i]\to\widetilde{U}_i$ are geodesic paths satisfying $\varphi_i\circ\tilde{c}_i =c|_{[t_{i-1},t_i]}$,
\item[$\cdot$] $h_i\colon U_i\to V_i$ are changes of charts at $\tilde{c}_{i}(t_{i})$ such that $h_i\colon \tilde{c}_{i}(t_i)\mapsto\tilde{c}_{i+1}(t_i)$ and $Dh_i\colon  \dot{\tilde{c}}_{i}(t_i)\mapsto\dot{\tilde{c}}_{i+1}(t_i)$ for $i=1,\ldots,k-1$, and $h_k$ is a change of charts at $\tilde{c}_k(b)$.
\end{itemize}

\begin{figure}[h]
\begin{center}
\leavevmode\hbox{}
\includegraphics[width=10cm]{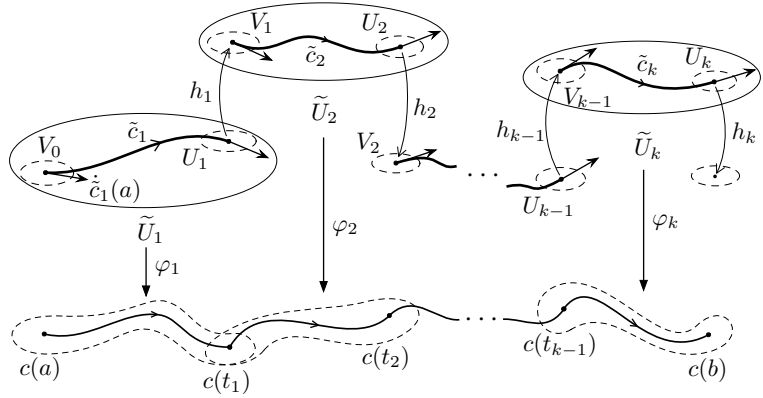}
\caption[Orbifold path]{Geodesic path over a subdivision $a=t_0\le t_1\le \cdots\le t_k=b$.}
\label{fig:orb_path}
\end{center}
\end{figure}

Among geodesic paths with underlying continuous map $c\colon [a,b]\to Q$ we define an equivalence relation generated by the following two operations \cite[III.$\mathcal{G}$.3.2]{Bri:99Met:aa}:
\begin{itemize}
\item[$(i)$] add a new point $t_i'\in[t_{i-1},t_i]$ to the subdivision $a=t_0\le t_1\le\dots\le t_k=b$, and replace the geodesic $\tilde{c}_i\colon [t_{i-1},t_i]\to \widetilde{U}_i$ by $\tilde{c}_i',1,\tilde{c}_i''$, where $\tilde{c}_i'=\tilde{c}_i|_{[t_{i-1},t_i']}$ and $\tilde{c}_i''=\tilde{c}_i|_{[t_i',t_i]}$, and $1\in\Gamma_{\tilde{c}_i(t')}$ is the identity element;
\item[$(ii)$] replace $(\tilde{c}_1, h_1,\tilde{c}_2,\ldots, \tilde{c}_k, h_k)$ by a new sequence $(\tilde{c}'_1, h'_1,\tilde{c}_2',\ldots,\tilde{c}'_k, h'_k)$ defined over the same subdivision and satisfying the following property: for each $i=1,\ldots,k$ there exist a change of charts $g_i$ defined from a neighbourhood of $\tilde{c}_i([t_{i-1},t_i])$ onto a neighbourhood of $\tilde{c}'_i([t_{i-1},t_i])$ such that $g_i\circ\tilde{c}_i=\tilde{c}'_i,$ $g_1^{-1}\circ\tilde{c}_1'$ and $\tilde{c}_1$ have the same germ at $a$, and for each $i=1,\ldots,k$ the maps $h_i'\circ g_{i}$ and $g_{i+1}\circ h_i$ have the same germ at $\tilde{c}_i(t_i)$.
\end{itemize}

Two geodesic paths over $c\colon [a,b]\to Q$ are said to be \emph{equivalent} if once redefined over a suitable common subdivision using operations of type $(i)$, one can pass from one to the other by an operation of type $(ii)$. 

\begin{defn}\label{def:geodesic}
An \emph{orbifold geodesic segment} on a Riemannian orbifold $\mathcal{Q}$ is a continuous path $c\colon [a,b]\to Q$ together with an equivalence class $[c]$ of geodesic paths over $c$. 
\end{defn}

For short we will denote an orbifold geodesic by $[c]\colon [a,b]\to \mathcal{Q}$. The points $x=c(a)$  and $y=c(b)$ are the initial and terminal points of $[c]$ in $Q$, and we say that the orbifold geodesic $[c]$ joins $x$ to $y$. 

The length of a geodesic path $(\tilde{c}_1, h_1,\tilde{c}_2,\ldots, \tilde{c}_k,h_k)$ over $c\colon [a,b]\to Q$ is defined to be the sum of the lengths of the geodesics $\tilde{c}_i$ for $i=1,\ldots,k$. Because the maps $g_i$ in an operation of type $(ii)$ are isometries, equivalent geodesic paths have the same length and, we define the length of an orbifold geodesic $[c]\colon [a,b]\to \mathcal{Q}$ to be length of an orbifold path $(\tilde{c}_1, h_1,\tilde{c}_2,\ldots, \tilde{c}_k,h_k)$ representing it. If the underlying map $c\colon [a,b]\to Q$ is the constant map at a point $x\in Q$, then any orbifold geodesic $[c]$ over $c$ has zero length and is represented by a pair $(x,\gamma)$, with $\gamma\in\Gamma_x$. 

An orbifold geodesic $[c]\colon [a,b]\to\mathcal{Q}$ is \emph{closed} if $c(a)=c(b)$ and the  change of charts $h_k\colon U_k\to V_0$ maps onto an open neighbourhood $V_0$ of $\tilde{c}_1(a)$ such that $$h_k\colon \tilde{c}_k(b)\mapsto\tilde{c}_1(a)\mbox{ and } Dh_k\colon \dot{\tilde{c}}_k(b)\mapsto\dot{\tilde{c}}_1(a).$$
\begin{figure}[h]
\begin{center}
\leavevmode\hbox{}
\includegraphics[width=7cm]{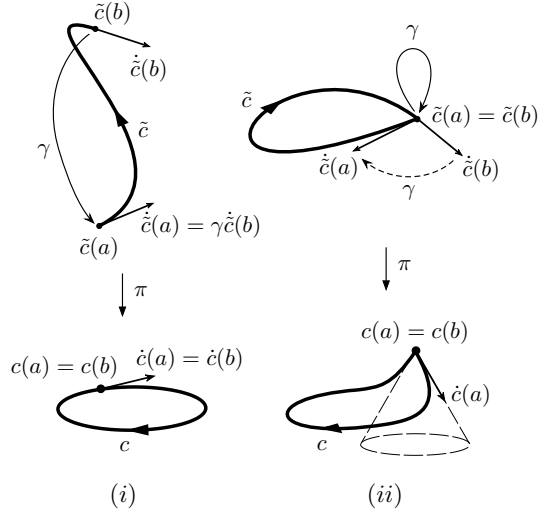}
\caption{Closed geodesics on $\mathcal{Q}=M/\Gamma$.}
\label{fig:geodesics}
\end{center}
\end{figure}

On developable orbifolds, geodesics have a much simpler form. Let $\mathcal{Q}=M/\Gamma$ be the orbifold quotient of a Riemannian manifold $M$ by the proper action of the discrete group of isometries $\Gamma\subseteq{\rm Isom}(M)$, and let $\pi\colon M\to Q$ be the quotient map. An orbifold geodesic $[c]:[a,b]\to Q$ over the subdivision $a=t_0\le t_1\le\dots\le t_k=b$ is given by a sequence $(\tilde{c}_1, \gamma_1,\tilde{c}_2,\ldots, \tilde{c}_k,\gamma_k)$, where $\tilde{c}_i$ are geodesics on the manifold $M$ and $\gamma_i\in\Gamma$ are isometries of $M$ (cf. \S\ \ref{ssec:good_orbifolds}). Define $\tilde{c}\colon [a,b]\to M$ to be the path in $M$ given by 
\[
\tilde{c}(t)=
\left\{
\begin{array}{ll}
\tilde{c}_1(t) & \mbox{ if } t\in[t_0,t_1] \\
\gamma_1^{-1}\tilde{c}_2(t) & \mbox{ if } t\in[t_1,t_2]\\
\gamma_1^{-1}\ldots\gamma_{i-1}^{-1}\tilde{c}_i(t) & \mbox{ if } t\in[t_{i-1},t_i] \mbox{ for } i=3,\ldots,k.
\end{array}
\right.
\] 
Since each $\gamma_i$ is an isometry, the restrictions $\tilde{c}|_{[t_{i-1},t_i]}$ are geodesics in $M$, and because each $\gamma_i^{-1}$ maps ${\tilde{c}}_{i+1}(t_i)$ to ${\tilde{c}}_i(t_i)$ and $\dot{\tilde{c}}_{i+1}(t_i)$ to $\dot{\tilde{c}}_i(t_i)$ for $i=1,\ldots,k-1$, the path $\tilde{c}\colon [a,b]\to M$ is a smooth geodesic in $M$. By letting $\gamma=\gamma_{k}\ldots\gamma_2\gamma_1$, we see that the geodesic path $(\tilde{c},\gamma)$  is equivalent to $(\tilde{c}_1, \gamma_1,\tilde{c}_2,\ldots, \tilde{c}_k,\gamma_k)$, and therefore the orbifold geodesic $[c]\colon [a,b]\to\mathcal{Q}$ can be represented by the pair $(\tilde{c},\gamma)$.

In conclusion, on a developable Riemannian orbifold $\mathcal{Q}=M/\Gamma$, geodesics are in one-to-one correspondence with equivalence classes of pairs $(\tilde{c},\gamma)$, where $\tilde{c}$ is a geodesic in $M$ and $\gamma\in \Gamma$.  Two pairs $(\tilde{c},\gamma)$ and $(\tilde{c}',\gamma')$ are equivalent if and only if there exists an element $\delta\in \Gamma$ such that $\tilde{c'}=\delta.\tilde{c}$ and $\gamma'=\delta^{-1}\gamma\delta$. An orbifold geodesic $(\tilde{c},\gamma)$ on $\mathcal{Q}$ is closed if 
$\gamma\tilde{c}(b)=\tilde{c}(a) \mbox{ and } \gamma\dot{\tilde{c}}(b)=\dot{\tilde{c}}(a)$ 
(see Figure \ref{fig:geodesics}).

\subsection{Orbifold Tangent Bundle}\label{ssec:tangent_bundle}
Let $\mathcal{Q}=(Q,\mathcal{A})$ be an effective $n$-orbifold with atlas $\mathcal{A} = \{(\widetilde{U}_i,\Gamma_i,\varphi_i)\}_{i\in\mathcal{I}}$. For each open set $U_i\subseteq Q$ with uniformizing chart $(\widetilde{U}_i,\Gamma_i,\varphi_i),$ the action of the finite group $\Gamma_i$ on $\widetilde{U}_i$ extends to a smooth action on the tangent bundle $T\widetilde{U}_i$ defined by  $$\gamma.(\tilde{x},\tilde\v)=\big(\gamma.\tilde{x},D_{\tilde{x}}\gamma(\tilde\v)\big),$$  where $\gamma\in\Gamma_i$,  $(\tilde{x},\tilde\v)\in T\widetilde{U}_i$ and $D_{\tilde{x}}\gamma\colon T_{\tilde{x}}\widetilde{U}_i\to T_{\gamma\tilde{x}}\widetilde{U}_i$ is the differential of $\gamma$ at $\tilde{x}.$ The orbit space $TU_i=T\widetilde{U}_i/\Gamma_i$ of the $\Gamma_i$-action on $T\widetilde{U}_i$ has a natural structure of a $2n$-dimensional orbifold uniformized by $(T\widetilde{U}_i,\Gamma_i,\pi_i)$, where $\pi_i\colon T\widetilde{U}_i\to TU_i$ is the quotient map.

We use the change of charts of the orbifold atlas $\mathcal{A}$ to glue together the local orbifolds $T\widetilde{U}_i/\Gamma_i$ and obtain the orbifold tangent bundle of $\mathcal{Q}$. Suppose $(\widetilde{U}_i,\Gamma_i,\varphi_i)$ and $(\widetilde{U}_j,\Gamma_j,\varphi_j)$ are two orbifold charts in $\mathcal{A}$ and let $\tilde{x}_i\in \widetilde{U}_i$ and $\tilde{x}_j\in \widetilde{U}_j$ be such that $\varphi_i(\tilde{x}_i)=\varphi_j(\tilde{x}_j)$. Then there exists a change of charts $h\colon \widetilde{U}\to \widetilde{V}$ defined on an open neighbourhood $\widetilde{U}$ of $\tilde{x}_i$ onto a neighbourhood $\widetilde{V}$ of $\tilde{x}_j$ such that $\varphi_j\circ h=\varphi_i|_{\widetilde{U}}$. The change of charts $h$ gives rise to a diffeomorphism $Dh\colon T\widetilde{U}\to T\widetilde{V}$ defined by $$Dh \colon (\tilde{x},\tilde\v) \mapsto (h(\tilde{x}),D_{\tilde{x}}h(\tilde\v)),$$ where $(\tilde{x},\tilde\v)\in T\widetilde{U}$ and $D_{\tilde{x}}h\colon T_{\tilde{x}}\widetilde{U}\to T_{h(\tilde{x})}\widetilde{V}$ is the differential of $h$ at $\tilde{x}$. Because $h$ is equivariant with respect to the local group actions on $\widetilde{U}$ and $\widetilde{V}$ (since $\varphi_j\circ h=\varphi_i|_{\widetilde{U}}$), the map $Dh\colon T\widetilde{U}\to T\widetilde{V}$ is an equivariant diffeomorphism with respect to the induced local group actions on $T\widetilde{U}$ and $T\widetilde{V}$. Therefore $\pi_j\circ Dh=\pi_i|_{T\widetilde{U}}$, and this allows us to identify $\pi_i(\tilde{x},\tilde\v)\sim\pi_j(\tilde{y},\tilde\w)$ in $TU$ whenever $Dh(\tilde{x},\tilde\v)=(\tilde{y},\tilde\w)$. This identification induces an equivalence relation $\sim$ on the disjoint union of the spaces $TU_i=T\widetilde{U}_i/\Gamma_i$ for $i\in \mathcal{I},$ and we define $TQ$ to be the quotient space $$TQ= \bigsqcup\limits_{i\in\mathcal{I}}TU_i\big{/}\sim$$ with the quotient topology. 
  
The collection of orbifold charts $\{(T\widetilde{U}_i,\Gamma_i,\pi_i)\}_{i\in\mathcal{I}}$ together with the change of charts $Dh$ defines on the space $TQ$ a $2n$-dimensional orbifold atlas, which we denote by $T\mathcal{A}$. We note that equivalent orbifold atlases on $Q$ define equivalent atlases on $TQ$ and give the following definition (see \cite[Proposition 1.21]{Ade:07Orb:aa}).
 
\begin{defn} 
 The \emph{orbifold tangent bundle} of the $n$-orbifold $\mathcal{Q}=(Q,\mathcal{A})$ is defined to be the $2n$-dimensional orbifold $T\mathcal{Q}=(TQ,T\mathcal{A})$.  
\end{defn} 

The natural projection onto the first factor $p\colon TQ\to Q$ defines a smooth orbifold map from the orbifold tangent bundle $T\mathcal{Q}$ onto $\mathcal{Q}$, and, unlike the tangent bundle of a manifold, the fibers of this map need not have the structure of a vector space, but rather that of a quotient of a vector space by the linear action of a finite group. To see this, choose $x\in Q$ and let $(\widetilde{U},\Gamma,\varphi)$ be an orbifold chart over an open set $U\subseteq Q$ containing $x$, and $(T\widetilde{U},\Gamma,\pi)$ be the corresponding local chart over $TU$ in the orbifold tangent bundle $T\mathcal{Q}$. The bundle projection $T\widetilde{U}\to \widetilde{U}$ is $\Gamma$-invariant and  induces a natural projection $p\colon TU\to U$ given by $p(\pi(\tilde{x},\tilde{\v})) = \varphi(\tilde{x})$, where $(\tilde{x},\tilde{\v})\in T\widetilde{U}$. Hence, the fiber $p^{-1}(x)$ is the image via the quotient map $\pi$ of the orbit $\Gamma.(\tilde{x},\tilde{v})\subset T\widetilde{U}$ with $\tilde{x}\in\varphi^{-1}(x)$ and $\tilde{\v}\in T_{\tilde{x}}\widetilde{U}$. Note that $\varphi^{-1}(x)\in\widetilde{U}$ is in bijective correspondence with the left cosets $\Gamma/\Gamma_x$ and that at each $\tilde{x}\in \varphi^{-1}(x)$ the isotropy group $\Gamma_{\tilde{x}}$ of $\tilde{x}$ acts linearly on the tangent space $T_{\tilde{x}}\widetilde{U}$. Moreover, if $\tilde{x}'$ is another point in $\varphi^{-1}(x)$, then $\Gamma_{\tilde{x}'}\simeq \Gamma_{\tilde{x}}$ and the quotient spaces $T_{\tilde{x}'}\widetilde{U}/\Gamma_{\tilde{x}'}$ and $T_{\tilde{x}}\widetilde{U}/\Gamma_{\tilde{x}}$ are naturally homeomorphic. Therefore $p^{-1}(x)=\{(x,\v)\mid \v\in T_{\tilde{x}}\widetilde{U}/\Gamma_{\tilde{x}}\} = \{x\}\times T_{\tilde{x}}\widetilde{U}/\Gamma_{\tilde{x}}$ for some $\tilde{x}\in\varphi^{-1}(x)$. 

In conclusion, the fiber $p^{-1}(x)\subset TU$ is homeomorphic to the quotient space $\R^n/\Gamma_x$, where the isotropy group $\Gamma_x$ is assumed to act on $\R^n$ via a faithful linear representation $\Gamma_x\subset GL_n(\R)$. 
 
The fiber $p^{-1}(x)\subset TQ$, denoted by $T_xQ$, is called the \emph{tangent cone to $Q$ at $x$}.  Note that $T_xQ\simeq \R^n/\Gamma_x$ has a natural vector space structure if and only if $x$ is a regular point of $\mathcal{Q}$.

\section{Stratification of the Singular Locus}\label{sec:stratification}

Given a smooth $n$-orbifold $\mathcal{Q}=(Q,\mathcal{A})$, the orbifold structure $\mathcal{Q}$ induces a canonical stratification of underlying topological space $Q$ where each stratum corresponds to a connected set of points in $Q$ having the same isotropy type \cite{Kaw:78The:aa}. This canonical stratification gives a decomposition of the space $Q$ into a disjoint union of smooth manifolds having dimensions less than or equal to $n$. Each point $x\in Q$ belongs to only one such smooth manifold, and if this manifold has dimension $k$, we say that the point $x$ has singular dimension $k$. We can rearrange the canonical stratification of $Q$ by orbit type into a stratification by singular dimension where each stratum $\Sigma_k$ is a union of smooth manifolds of dimension $k$ for $0\le k\le n$.  

In this section we give a description of this stratification and show that in a Riemannian orbifold, the singular strata $\Sigma_k$ are totally geodesic. The stratification by singular dimension of an orbifold is discussed in C. Seaton's thesis (see \cite{Sea:08Two:aa} for the paper version) and we borrow from the terminology introduced therein.

Suppose $\mathcal{Q}=(Q,\mathcal{A})$ is a smooth $n$-dimensional orbifold. Let $x\in Q$ and let $(\widetilde{U},\Gamma,\varphi)$ be an orbifold chart at $x$. Choose $\tilde{x}\in\widetilde{U}$ such that $\varphi(\tilde{x})=x$.  As noticed in \S\ \ref{ssec:tangent_bundle} above, the isotropy group $\Gamma_{\tilde{x}}$ acts linearly on the tangent space $T_{\tilde{x}}\widetilde{U}$.  Let $(T_{\tilde{x}}\widetilde{U})^{\Gamma_{\tilde{x}}}$ be the vector subspace of $T_{\tilde{x}}\widetilde{U}$ obtained as the fixed-point set of this $\Gamma_{\tilde{x}}$-action. If $h\colon \widetilde{U}\to\widetilde{V}$ is a change of charts at $x$, then $h$ is $\Gamma_x$-equivariant and the differential $Dh_{\tilde{x}}\colon T_{\tilde{x}}\widetilde{U}\to T_{h(\tilde{x})}\widetilde{V}$ induces an isomorphism between the subspaces fixed by  
the corresponding $\Gamma_x$-actions on $T_{\tilde{x}}\widetilde{U}$ and $T_{h(\tilde{x})}\widetilde{V}$. This shows that up to isomorphism, the vector subspace $(T_{\tilde{x}}\widetilde{U})^{\Gamma_{\tilde{x}}}$ is independent of both the choice of the orbifold chart  $\widetilde{U}$ at $x$ and the lift of $x$ in $\widetilde{U}$. 

 The isomorphism class of the vector subspace $(T_{\tilde{x}}\widetilde{U})^{\Gamma_{\tilde{x}}}$ is denoted $T_xQ^{\Gamma_x}$ and is called the {\it space of tangent vectors} at $x$.  The dimension of this vector space depends only on the local action of the isotropy group $\Gamma_x$ at $x$, and we give the following definition.

\begin{defn}
The \emph{singular dimension} of a point $x$ of an effective orbifold $\mathcal{Q}$ is defined to be the dimension of the space of tangent vectors $T_xQ^{\Gamma_x}$ at $x$.
\end{defn}

For each $k=0,\ldots, n$, we let $\Sigma_k$ denote the set of points in $Q$ with singular dimension $k$. Since each point $x\in Q$ belongs to only one stratum $\Sigma_k$ for some $k$, the underlying space $Q$ can be written as the disjoint union $Q= \bigsqcup\limits_{k=0}^{n}\Sigma_k$. We refer to this decomposition as the \emph{stratification by singular dimension} of the underlying topological space $Q$. It depends only on the orbifold structure $\mathcal{Q}$ and not the particular choice of an atlas defining $\mathcal{Q}$.

\begin{rem}\label{rem:regular} 
Note that $Q_{\rm reg}=\Sigma_n$ and the singular locus $\Sigma=\bigsqcup\limits_{k=0}^{n-1}\Sigma_k.$
\end{rem}

Proposition \ref{prop:strata} below shows that for each $k=0,\ldots,n$, the stratum $\Sigma_k$ can be given the structure of a $k$-dimensional manifold; and when $\mathcal{Q}$ is a Riemannian orbifold, each connected component of $\Sigma_k$ is totally geodesic in $\mathcal{Q}$. Recall that a submanifold $N\subset M$ of a Riemannian manifold $M$ is \emph{totally geodesic} if any geodesic in $N$ with the induced Riemannian metric from $M$ is also a geodesic in $M$. Alternatively, $N$ is totally geodesic submanifold of $M$ if for each $x\in N$ there exists a neighbourhood of the origin in $T_xN$ which is mapped onto $N$ via the exponential map $\exp_x\colon T_xM\to M$. 
  
 We first state an elementary result concerning the structure of the fixed point set of a family of isometries.

\begin{lemma}[{\cite[Proposition 24]{Pet:06Rie:aa}}]\label{lemma:fixed_point_set}
Suppose $M$ is a complete Riemannian manifold and $\Lambda\subseteq {\rm Isom}(M)$ is a set of isometries of $M$. Let $F$ be the set of points of $M$ which are fixed by all the elements of $\Lambda$. Then each connected component of $F$ is a totally geodesic submanifold of $M$.
\end{lemma}

\begin{proof}
Suppose $F\neq\varnothing$ and let $x\in F$. Let $V$ be the maximal vector subspace of $T_xM$ which is fixed by all the linear isometries $\{ D_x\gamma\colon T_xM\to T_xM \mid\gamma\in\Lambda\}$.  Since $M$ is complete, the exponential map $\exp_x\colon T_xM\to M$ is defined for all $\v\in T_xM$ and if $\gamma\in \Lambda$, then $\exp_x(D_x\gamma(\v))=\gamma(\exp_x(\v))$. In particular, if $\v\in V$, then $\exp_x(\v) = \exp_x(D_x\gamma(\v)) = \gamma(\exp_x(\v))$ for all $\gamma\in \Lambda$, and this shows that $\exp_x\colon V\to F$. Let $B\in T_xM$ be an open ball centred at the origin in $T_xM$, such that the restriction $\exp_x|_B\colon B\to M$ is a diffeomorphism onto its image. We want to show that $\exp_x(V\cap B)=F\cap\exp_x(B)$. Let $y\in F\cap\exp_x(B)$ and let $c\colon [0,1]\to M$ be the unique geodesic joining $x$ to $y$. That is, $c(t)=\exp_x(t\v)$, where $\v=\exp_x^{-1}(y)\in B$. If $\gamma\in\Lambda$, then $t\mapsto\exp_x(D_x\gamma(t\v))=\gamma(c(t)), t\in[0,1]$ is also a geodesic connecting $x=\gamma(x)$ to $y=\gamma(y)$, and because of the uniqueness of the geodesic from $x$ to $y$, it follows that $\gamma(c(t))=c(t)$ for all $t\in[0,1]$. Hence $\v=\exp_x^{-1}(y)$ is fixed by $D_x\gamma$ and, since this happens for all $\gamma\in\Lambda$, it implies that $\exp_x^{-1}(y)\in V\cap B$. This shows that  $\exp_x\colon V\cap B\to F\cap \exp_x(B)$ is a diffeomorphism. In conclusion, the open neighbourhood $F\cap\exp_x(B)$ of $x$ in $F$ has the structure of a totally geodesic $k$-submanifold of $M$, with $k=\dim(V)$. Hence $F$ consists of boundary-less totally geodesic submanifolds of $M$. 
\end{proof}

\begin{prop}\label{prop:strata}
Suppose $\mathcal{Q}$ is an effective $n$-dimensional orbifold without boundary. For each $k=0,\ldots,n$, the set $\Sigma_k$ has naturally the structure of a $k$-dimensional manifold without boundary. The tangent space $T_x\Sigma_k$ at a point $x\in \Sigma_k$ is canonically identified with $T_xQ^{\Gamma_x}$, the space of tangent vectors at $x$. Furthermore, if $\mathcal{Q}$ is a Riemannian orbifold, then each connected component of $\Sigma_k, k\ge 1$ is totally geodesic in $\mathcal{Q}$. 
\end{prop}

\begin{proof} 
Assume that the orbifold $\mathcal{Q}$ has a Riemannian structure (see Remark \ref{rem:riemann_orb}). Note that if $k=0$, then $\Sigma_0$ consists of a discrete collection of points and therefore $\Sigma_0$ has the structure of a zero-dimensional manifold. 

Fix now $k\in\{1,\ldots,n\}$ such that $\Sigma_k\neq\varnothing$. Let $x\in \Sigma_k$ and let $U_x$ be a fundamental neighbourhood at $x$ with fundamental chart $(\widetilde{U}_x,\Gamma_x,\varphi_x)$. That is, $\widetilde{U}_x$ is a connected complete Riemannian manifold and the isotropy group $\Gamma_x$ of $x$ acts properly by isometries on $\widetilde{U}_x$ such that $U_x=\widetilde{U}_x/\Gamma_x$ (see Remark \ref{rem:fundam_nbhd}). 

Let $\widetilde{\Sigma}_{\Gamma_x}$ be the fixed point set of $\Gamma_x$ in $\widetilde{U}_x$ and let $\tilde{x}\in \widetilde{\Sigma}_{\Gamma_x}$ such that $\varphi_x(\tilde{x})=x$.  By Lemma \ref{lemma:fixed_point_set}, the set $\widetilde{\Sigma}_{\Gamma_x}$ has the structure of a totally geodesic submanifold of $\widetilde{U}_x$ of dimension $k={\rm dim}((T_{\tilde{x}}\widetilde{U}_x)^{\Gamma_x})$.  Moreover, since $(\widetilde{U}_x,\Gamma_x,\varphi_x)$ is a fundamental orbifold chart, the manifold $\widetilde{\Sigma}_{\Gamma_x}$ is connected and the restriction $\varphi_x|_{\widetilde{\Sigma}_{\Gamma_x}}\colon \widetilde{\Sigma}_{\Gamma_x}\to \Sigma_k\cap U_x$ is injective onto its image. Thus  $\varphi_x|_{\widetilde{\Sigma}_{\Gamma_x}}$ induces a homeomorphism from $\widetilde{\Sigma}_{\Gamma_x}$ onto the open neighbourhood $\Sigma_k\cap U_x$ of $x$ in $\Sigma_k$. 

It is clear from the construction that such open charts $(\Sigma_k\cap U_x, \varphi_x|_{\widetilde{\Sigma}_{\Gamma_x}})$ exist at each point $x\in\Sigma_k$, and we can use the change of charts of the orbifold $\mathcal{Q}$ to patch together these local charts. 
Let $z\in \Sigma_k\cap U_x\cap U_y$ for some $x$ and $y$ in $\Sigma_k$ with fundamental neighbourhoods $U_x$ and $U_y$. Then $\varphi_x^{-1}(z)\in\widetilde{\Sigma}_{\Gamma_x}$, $\varphi_y^{-1}(z)\in\widetilde{\Sigma}_{\Gamma_y}$ and the isotropy groups $\Gamma_x$, $\Gamma_y$ and $\Gamma_z$ are all isomorphic. If $h\colon \widetilde{U}\to \widetilde{V}$ is a change of orbifold charts from an open connected neighbourhood $\widetilde{U}$ of $\varphi_x^{-1}(z)$ onto a neighbourhood $\widetilde{V}$ of $\varphi_y^{-1}(z)$ such that $\varphi_y\circ h=\varphi_x|_{\widetilde{U}}$, then the restriction $h|_{\widetilde{\Sigma}_{\Gamma_x}\cap\widetilde{U}}\colon \widetilde{\Sigma}_{\Gamma_x}\cap\widetilde{U}\to\widetilde{\Sigma}_{\Gamma_y}\cap\widetilde{V}$ is a smooth Riemannian isometry which gives a change charts between $(\Sigma_k\cap U_x, \varphi_x|_{\widetilde{\Sigma}_{\Gamma_x}})$ and $(\Sigma_k\cap U_y,\varphi_y|_{\widetilde{\Sigma}_{\Gamma_y}})$. The collection of all such charts together with the corresponding change of charts as above give $\Sigma_k$ the structure of a Riemannian $k$-dimensional manifold without boundary. 

We can proceed as in the proof of Lemma \ref{lemma:fixed_point_set} and show that for each $x\in \Sigma_k$, the composition $\varphi_x\circ\exp_{\tilde{x}}\colon T_{\tilde{x}}\widetilde{U}_x\to U_x$ restricts to a diffeomorphism from a neighbourhood of $0\in (T_{\tilde{x}}\widetilde{U}_x)^{\Gamma_x}$  onto a neighbourhood of $x$ in $\Sigma_k\cap U_x$. This shows that the connected components of $\Sigma_k$ are totally geodesic in $\mathcal{Q}$. 
\end{proof}

If $S\subseteq\Sigma_k$ is a connected component of singular dimension $k$, then the isotropy groups of any two points in $S$ are naturally isomorphic; and we denote by $\Gamma_S$ the group representing the isomorphism class of the isotropy along the component $S$.  For a developable orbifold $\mathcal{Q}=M/\Gamma$, this isomorphism class  is just the conjugacy class of $\Gamma_x$ in $\Gamma$, for some $x\in S$. We note, however, that points with the same singular dimension $k$ but contained in different connected components of $\Sigma_k$ need not have isomorphic isotropy groups. For example, both singular points in the $\Z_p\shy\Z_q$-football orbifold with $p\neq q$ belong to $\Sigma_0$, but they have non-isomorphic isotropy groups $\Z_p$ and $\Z_q$, respectively (see for instance \cite[Example 2.37]{Dra:11Clo:aa}). 

\begin{rem}\label{rem:strata}
If $x\in\Sigma_k$ and $U_x$ is a fundamental neighbourhood at $x$ with uniformizing chart $(\widetilde{U}_x,\Gamma_x,\varphi_x)$,  then $\Sigma_k\cap U_x$ is connected and $k$ is the smallest singular dimension in $U_x$. Indeed, if $\Sigma_\ell\cap U_x\neq\varnothing$ and $x\notin\Sigma_\ell$, then $\ell>k$ and the isotropy along the connected components of $\Sigma_\ell$ corresponds to the conjugacy class of a proper subgroup of $\Gamma_x$ (see Remark \ref{rem:fundam_nbhd}). Conversely, for each proper subgroup $\Gamma_x'$ of $\Gamma_x$, the fixed point set $\widetilde{\Sigma}_{\Gamma_x'}$ of $\Gamma'_x$ in $\widetilde{U}_x$ has dimension $\ell\ge k$, and if $\ell>k$, then the corresponding component of $\Sigma_\ell\cap U_x$ will have $x$ in its closure. If a subgroup $\Gamma_x''$ is conjugate to $\Gamma_x'$ in $\Gamma_x$, then the fixed point set $\widetilde{\Sigma}_{\Gamma_x''}$ has dimension $\ell$ and projects to the same component in $\Sigma_\ell$ as $\widetilde{\Sigma}_{\Gamma_x'}$. With the notation above, when $k=0$ and $\ell=1$, it is possible that there are two distinct connected components of $\Sigma_1\cap U_x$ corresponding to the conjugacy class of $\Gamma'_x$.
 \end{rem}

In a compact orbifold $\mathcal{Q}$ the number of connected components of each of the manifolds $\Sigma_k$ is finite. 

\begin{prop}\label{prop:finitely_many_components}
If $\mathcal{Q}$ is an effective compact connected $n$-orbifold, then for each $k = 0,\ldots,n,$ the manifolds $\Sigma_k$ have finitely many connected components.  
\end{prop}

\begin{proof}
 Note first that $\Sigma_0$ is a discrete subset of $Q$, and since $Q$ is compact, $\Sigma_0$ is finite. Let $\{U_{x_1},\ldots,U_{x_j}\}$ be a finite open cover of $Q$ by fundamental neighbourhoods, where each $U_{x_i}$ is uniformized by a fundamental chart $(\widetilde{U}_{x_i},\Gamma_{x_i},\varphi_{x_i})$ at $x_i$. Such a finite open cover exists because $Q$ is compact. Note that $\Sigma_0\subseteq \{x_1,\ldots,x_j\}.$ For each $x_i$, the number of connected components of $\Sigma_k\cap U_{x_i}$ is bounded from above by the number of conjugacy classes of  subgroups of $\Gamma_{x_i}$ when $k\ge 2$, and by twice this number when $k=1$ (cf. Remark \ref{rem:strata}). Since each isotropy group $\Gamma_{x_i}$ is finite, the number of singular components in each $U_{x_i}$ is finite, and the conclusion follows. 
\end{proof}

Let $S$ be a connected component of $\Sigma_k$ for some fixed $k$ with $0\le k\le n$, where $n$ is the dimension of the orbifold $\mathcal{Q}$. We define the \emph{frontier of $S$} to be the set $\fr(S)$ consisting of points in $Q\smallsetminus S$ which are limit points of sequences in $S$; and define the \emph{closure of $S$} to be the union $\cl(S)=S\cup\fr(S)$. If $\fr(S)=\varnothing$, then $\cl(S)=S$ and we say that the component $S$ is \emph{closed}. 

While the components of $\Sigma_0$ are closed (since $\Sigma_0$ is a discrete set of points), this is not necessarily true for connected components with singular dimension $k>0$.  If $S\subseteq\Sigma_k$ is such an open component, then the connected components of the frontier of $S$ have singular dimensions $<k$ and larger isotropy groups, having proper subgroups isomorphic to $\Gamma_S$ (cf. Remark \ref{rem:strata}). As a particular case, when $k=n$, the stratum $\Sigma_n=Q_{\rm reg}$ is connected, open and dense in $Q$. In this case, the frontier $\fr(\Sigma_n)=Q\smallsetminus Q_{\rm reg}=\Sigma$ is the singular locus, and the closure $\cl(\Sigma_n)=Q$ is the underlying topological space of the orbifold $\mathcal{Q}$.
 
In general, the closure $\cl(S)$ of an open connected component $S\subseteq\Sigma_k$ with $0<k<n$ need not have the structure of an orbifold. However, if $k>0$ is the {\it smallest} singular dimension such that $\Sigma_k\neq\varnothing$, then the closure of the connected components of $\Sigma_k$ can be given a natural non-effective orbifold structure.  In the following theorem we prove this in the case when $\Sigma_1=\varnothing$, which implies that the smallest positive singular dimension $k$ for the orbifold $\mathcal{Q}$ is at least two. The case when $\Sigma_1\neq\varnothing$ will be considered in Theorem \ref{thm:singular_1}. 

\begin{theorem}\label{thm:closed_stratum}
Suppose $\mathcal{Q}$ is a compact effective $n$-orbifold such that  $\Sigma_1=\varnothing$. Let $k>0$ be the smallest singular dimension such that $\Sigma_k\neq\varnothing$. Then the closure $\cl(S)$ of each connected component $S\subseteq \Sigma_k$ has a natural structure of a compact $k$-orbifold $\mathcal{S}$ such that the associated effective orbifold $\mathcal{S}_{\rm eff}$ has only zero-dimensional singular locus or is a smooth manifold. If $\mathcal{Q}$ is a Riemannian orbifold, then $\mathcal{S}$ is totally geodesic in $\mathcal{Q}$.
\end{theorem}

\begin{proof}
Let $\mathcal{Q}$ be as in the statement of the theorem. Since $\Sigma_1=\varnothing$, the smallest positive singular dimension $k$ such that $\Sigma_k\neq\varnothing$ is at least 2. If $k=n$, then $S=Q_{\rm reg}$ and the conclusion follows since $\cl(S)=Q$ and $\mathcal{S}=\mathcal{Q}$.

Fix $k\in\{2,\ldots, n-1\}$ and let $S\subseteq \Sigma_k$ be a connected component of singular dimension $k$. Let $\Gamma_S$ denote the isotropy  along the component $S$.

When $S$ is closed, by Proposition \ref{prop:strata}, $\cl(S)=S$ has the structure of a $k$-dimensional manifold which is compact, since $\mathcal{Q}$ is compact; and if $\mathcal{Q}$ is a Riemannian orbifold, then $S$ is totally geodesic in $\mathcal{Q}$.

Assume now that $S$ is open. The points in the frontier of $S$ belong to singular strata of dimension $<k$ and since $k>0$ is the smallest positive singular dimension, i.e. $\Sigma_\ell=\varnothing$ for all $0<\ell<k$, it follows that $\fr(S)\subseteq\Sigma_0$. Moreover, because $\mathcal{Q}$ is compact, by Proposition \ref{prop:finitely_many_components}, the set $\Sigma_0$ is finite. Thus $\fr(S)$ consists of a finite set of points with singular dimension zero. 

Fix $x\in \fr(S)$. Let $(\widetilde{U}_x,\Gamma_x,\varphi_x)$ be a fundamental orbifold chart at $x$ and let $\tilde{x}=\varphi_x^{-1}(x)\in \widetilde{U}_x$. Denote by $S_x^\circ=S\cap \varphi_x(\widetilde{U}_x)$ and $S_x=S_x^\circ\cup\{x\}$. Since $k\ge2$, $S_x^\circ$ is connected ($S_x^\circ$ is homeomorphic to a punctured $k$-disk). 

Let $\widetilde{S}_x^\circ$ be a connected component in $\varphi_x^{-1}(S_x^\circ)$ and let $\Gamma\le\Gamma_x$ be the isotropy group along $\widetilde{S}_x^\circ$. Clearly $\Gamma\simeq\Gamma_S$, the isotropy along $S$. Then $\widetilde{S}_x=\widetilde{S}_x^\circ\cup\tilde{x}$ is the fixed point set of $\Gamma$ in $\widetilde{U}_x$ and by Proposition \ref{prop:strata}, $\widetilde{S}_x$ is a $k$-dimensional submanifold of $\widetilde{U}_x$. The action of $\Gamma_x$ on $\widetilde{U}_x$ restricts to $\widetilde{S}_x$ to a non-effective smooth action by $N_{\Gamma_x}(\Gamma)$, the normalizer of $\Gamma$ in $\Gamma_x$, and the kernel of this action is $\Gamma$. Moreover, the restriction $\varphi_x|_{\widetilde{S}_x}$ induces a homeomorphism from the quotient space $\widetilde{S}_x/N_{\Gamma_x}(\Gamma)$ of this action onto the open set $S_x\subset {\rm cl}(S)$. Then $(\widetilde{S}_x,N_{\Gamma_x}(\Gamma),\varphi_x|_{\widetilde{S}_x})$ defines a non-effective $k$-dimensional orbifold chart at $x$ over $S_x$. The corresponding reduced orbifold chart over $S_x$ is $(\widetilde{S}_x,N_{\Gamma_x}(\Gamma)/\Gamma,\varphi_x|_{\widetilde{S}_x})$, and in the associated effective orbifold structure on $S_x$ each point in $S_x^\circ$ has trivial isotropy and the point $x$ has isotropy group $N_{\Gamma_x}(\Gamma)/\Gamma$. Note that if $\Gamma$ is self normalizing in $\Gamma_x$, then $x$ is actually a smooth point of $S_x$ and in this case $S_x$ has the structure of a $k$-dimensional manifold.

The number of connected components $\widetilde{S}_x^\circ$ in $\varphi_x^{-1}(S_x^\circ)$ is in one to one correspondence with the number of conjugates $\gamma\Gamma\gamma^{-1}$ of $\Gamma$ in $\Gamma_x$, and it is thus equal to $|\Gamma_x\colon N_{\Gamma_x}(\Gamma)|$, the index of the normalizer of $\Gamma$ in $\Gamma_x$. In particular, if the subgroup $\Gamma$ is not normal in $\Gamma_x$, then the component $\widetilde{S}_x^\circ$ is not unique. However, the $k$-dimensional orbifold structure on $S_x$ is independent of the choice of the component $\widetilde{S}_x^\circ$ in $\varphi_x^{-1}(S_x^\circ)$. Indeed, if $\widetilde{S}_x^{\circ'}$ is a different choice of such a component, then there exists $\gamma\in\Gamma_x\smallsetminus N_{\Gamma_x}(\Gamma)$ such that $\widetilde{S}_x^{\circ'}=\gamma.\widetilde{S}_x^\circ$ and $\widetilde{S}_x'=\widetilde{S}_x^{\circ'}\cup \tilde{x}$ is the fixed point set of the conjugate $\Gamma'=\gamma\Gamma\gamma^{-1}$. Note that $N_{\Gamma_x}(\Gamma')=\gamma N_{\Gamma_x}(\Gamma)\gamma^{-1}$. The map $h=\gamma|_{\widetilde{S}_x}\colon \widetilde{S}_x\to\widetilde{S}_x'$ is a smooth diffeomorphism and gives a change of charts between $(\widetilde{S}_x,N_{\Gamma_x}(\Gamma),\varphi_x|_{\widetilde{S}_x})$ and $(\widetilde{S}_x',N_{\Gamma_x}(\Gamma'),\varphi_x|_{\widetilde{S}_x'})$. This shows that the two orbifold charts give the same orbifold structure on $S_x$.

Let now $y\in S$ and let $(\widetilde{U}_y,\Gamma_y,\varphi_y)$ be a fundamental orbifold chart at $y$. Denote $S_y=S\cap\varphi_y(\widetilde{U}_y)$ and let $\widetilde{S}_y=\varphi_y^{-1}(S_y)$. Then $\widetilde{S}_y={\rm Fix}(\Gamma_y)$ is a $k$-dimensional submanifold of $\widetilde{U}_y$ and the action of $\Gamma_y$ on $\widetilde{U}_y$ restricts to the trivial action on $\widetilde{S}_y$ with kernel $\Gamma_y\simeq\Gamma_S$. This gives a non-effective orbifold chart $(\widetilde{S}_y,\Gamma_y,\varphi_y|_{\widetilde{S}_y})$ over $S_y$ and thus making $S_y$ into a $k$-dimensional orbifold having isotropy $\Gamma_y$ at each of its points.

We can proceed as in the proof of Proposition \ref{prop:strata} and use restrictions of the change of charts of the orbifold $\mathcal{Q}$ to patch together these local orbifolds charts and obtain in this way a non-effective $k$-dimensional orbifold structure $\mathcal{S}$ on $\cl(S)$. In the associated effective orbifold $\mathcal{S}_{\rm eff}$ of $\mathcal{S}$,  we have that $S\subseteq (\mathcal{S}_{\rm eff})_{\rm reg}$ and the singular locus of $\mathcal{S}_{\rm eff}$, $\Sigma_\mathcal{S}\subseteq \fr(S)\subseteq \Sigma_0$. It is clear from construction of the local charts of $\mathcal{S}$ that if $\mathcal{Q}$ is a Riemannian orbifold, then $\mathcal{S}$ is totally geodesic in $\mathcal{Q}$.
\end{proof}

Theorem \ref{thm:closed_stratum} holds if $\mathcal{Q}$ is assumed to be a complete orbifold, and in this case, the orbifold $\mathcal{S}$ is a complete orbifold which is not necessarily compact. 

\section{Existence of geodesics}\label{sec:geodesics01}

Suppose $\mathcal{Q}$ is an effective compact Riemannian $n$-orbifold and let $Q=\bigsqcup\limits_{k=0}^n\Sigma_k$ be the stratification by singular dimension of $Q$. By Proposition \ref{prop:strata}, each connected component of $\Sigma_k$ has a natural structure of a Riemannian $k$-manifold which is totally geodesic in $\mathcal{Q}$; and when $k>0$, this implies that nontrivial closed geodesics contained in these $k$-manifolds give rise to closed geodesics of positive length on the orbifold $\mathcal{Q}$. In this section, we investigate sufficient conditions on the singular stratification of the orbifold $\mathcal{Q}$ which imply the existence of closed geodesics on $\mathcal{Q}$. 

To begin, we recall that any compact Riemannian manifold without boundary admits at least one closed geodesic of positive length \cite{Lyu:51Var:aa}. In a compact orbifold,  a closed connected component of $\Sigma_k$ is compact, and we have the following result.

\begin{theorem}\label{thm:closed_component}
Let $\mathcal{Q}$ be a compact Riemannian $n$-orbifold and assume that there exists a closed connected component of singular dimension $k$ for some $k>0$. Then $\mathcal{Q}$ admits at least one closed geodesic of positive length.
\end{theorem}

\begin{proof}
Let $S\subset\Sigma_k$ be a nonempty closed connected component of singular dimension $k>0$. Since $\mathcal{Q}$ is compact, by Proposition \ref{prop:strata}, $S$ has the structure of a compact $k$-dimensional Riemannian manifold without boundary which is totally geodesic in $\mathcal{Q}$. By the result of Lyusternik and Fet  in \cite{Lyu:51Var:aa}, $S$ admits a nontrivial closed geodesic which, in turn, gives rise to a closed orbifold geodesic of positive length in $\mathcal{Q}$.
\end{proof}

Note that for $k=n$ in Theorem \ref{thm:closed_component}, $\Sigma_n=Q_{\rm reg}$ is connected, and $\Sigma_n$ is closed if and only if the orbifold $\mathcal{Q}$ has the structure of a compact manifold. Note also that if $c\colon [0,1]\to \Sigma_k$ is a smooth parametrized geodesic contained in a connected component $S\subset \Sigma_k$, then the parametrized orbifold geodesics $[c]\colon [0,1]\to \mathcal{Q}$ covering $c$ are in one-to-one correspondence with the conjugacy classes of elements in $\Gamma_S$. In particular, for each parametrized geodesic $c\colon [0,1]\to Q_{\rm reg}$ there exists a unique orbifold geodesic $[c]\colon [0,1]\to \mathcal{Q}$ with underlying path $c$. 

The following corollary is a direct consequence of Theorem \ref{thm:closed_component}.

\begin{cor}\label{cor:no_zero_singularity}
Suppose $\mathcal{Q}$ is a compact connected Riemannian $n$-orbifold such that $\Sigma_0=\varnothing$. Then there exists at least one closed geodesic of positive length in $\mathcal{Q}$.  
\end{cor} 

\begin{proof}
Let $k>0$ be the smallest singular dimension such that $\Sigma_k\neq\varnothing$ and let $S$ be a connected component of $\Sigma_k$. The points in ${\rm fr}(S)$ have singular dimension $< k$, and since $\Sigma_\ell=\varnothing$ for all $0\le\ell<k$, the frontier ${\rm fr}(S)$ is empty. Thus $S$ is closed and the conclusion follows from Theorem \ref{thm:closed_component}.
\end{proof}

A particularly interesting situation is when $\Sigma_1\neq\varnothing$. In this case, again by Proposition \ref{prop:strata}, constant speed parametrizations of the connected components in $\Sigma_1$ define orbifold geodesic paths in $\mathcal{Q}$ and, as we show in Theorem \ref{thm:singular_1}, when $\mathcal{Q}$ is compact, $\cl(\Sigma_1)$ contains a totally geodesic compact $1$-orbifold without boundary which gives rise to a closed geodesic of positive length on the orbifold $\mathcal{Q}$. 

Recall that, up to diffeomorphism between orbifolds, the only compact connected one-dimensional effective orbifolds without boundary are the circle $\mathbb{S}^1$ (with trivial orbifold structure) and the orbifold $\mathcal{I}$, with underlying space the interval $[0,1]$ and singular points 0 and 1, each having isotropy $\Z_2$, the cyclic group of order two. Note that the circle $\mathbb{S}^1$ is an orbifold cover of the orbifold $\mathcal{I}$, and that both these one-dimensional orbifolds contain closed geodesics of positive length. 

Clearly, if $S$ is a closed component of $\Sigma_1$, then $S\simeq\mathbb{S}^1$ and any constant speed parametrization of $S$ gives a closed geodesic of positive length in $\mathcal{Q}$ (this also follows from Theorem \ref{thm:closed_component}). On the other hand, if $S$ is an open connected component of $\Sigma_1$, then $S\simeq(0,1)$ and $\cl(S)\simeq[0,1]$, but this time $\cl(S)$ might not inherit from $\mathcal{Q}$ a non-effective orbifold structure $\mathcal{S}$ so that $\mathcal{S}_{\rm eff}$ is diffeomorphic to the nontrivial compact orbifold $\mathcal{I}$ (for instance, this happens when the isotropy group of one of the frontier points of $S$ has odd order). We notice, however, that in this case, the component $S$ can be prolongated past its frontier to a possibly different component of $\Sigma_1$. We discuss this aspect in the following lemma. 

\begin{lemma}\label{lem:frontier_0}
Suppose $\mathcal{Q}$ is a connected Riemannian $n$-orbifold such that $\Sigma_1\neq\varnothing$. Assume $S\subseteq\Sigma_1$ is an open connected component and let $x\in\fr(S)$. Then either $S$ terminates at $x$ or $S$ extends smoothly past $x$ to a possibly different connected component in $\Sigma_1$.
\end{lemma}

\begin{proof} Let $S$ be an open connected component of $\Sigma_1$. Let $c\colon (0,1)\to S$ be a smooth constant speed parametrization of $S$. The frontier $\fr(S)\subset \Sigma_0$ is given by the limit points $c(0)=\lim\limits_{t\to 0^+}c(t)$ and $c(1)=\lim\limits_{t\to 1^-}c(t)$. 

Let $x=c(0)$ and let $(\widetilde{U}_x,\Gamma_x,\varphi_x)$ be a fundamental chart at $x$.  Let $\tilde{x}=\varphi_x^{-1}(x)$ be the lift of $x$ in $\widetilde{U}_x$. Let $\tilde{c}_0\colon [0,\varepsilon_0)\to \widetilde{U}_x$ be a geodesic path such that $\tilde{c}_0(0)=\tilde{x}$ and $\varphi_x\circ\tilde{c}_0=c|_{[0,\varepsilon_0)}$ for some $0<\varepsilon_0<1$. Since the image of $c|_{(0,\varepsilon_0)}$ is contained in the stratum $\Sigma_1$, there exists a proper subgroup $\Gamma_0\le \Gamma_x$ which fixes $\tilde{c}_0(0,\varepsilon_0)$ pointwise. Then $\Gamma_0$ fixes the $1$-dimensional vector subspace $V_0\subset T_{\tilde{x}}\widetilde{U}_x$ spanned by $\v_0=\dot{\tilde{c}}_0(0)$, and the geodesic $\tilde{c}_0\colon [0,\varepsilon_0)\to \widetilde{U}_x$ can be extended in $\widetilde{U}_x$ through $\tilde{x}$ to a $\Gamma_0$-invariant geodesic $\tilde{c}_0\colon (-\varepsilon_0,\varepsilon_0)\to \widetilde{U}_x$.  

We distinguish two possible situations. The first one is when $\Gamma_x$ contains an element $\gamma$ such that $\gamma\v_0=-\v_0$. Then 
$$\varphi_x(\tilde{c}_0(t))=\varphi_x(\tilde{c}_0(-t))=c(t) \mbox{ for all } 0\le t\le \varepsilon_0.$$ 
In this case we say that the component $S=c(0,1)\subset\Sigma_1$ terminates at $x=c(0)$, or that $x$ is an \emph{end of $S$}. The second possible situation is when there are no elements $\gamma\in\Gamma_x$ that satisfy $\gamma\v_0=-\v_0$. Then 
$$\varphi_x(\tilde{c}_0(t))\neq\varphi_x(\tilde{c}_0(-t)) \mbox{ for all } 0<t<\varepsilon_0.$$ 
In this case $S$ does not end at $x$, and we can prolongate $S$ to the left to a continuous path $c\colon (-\varepsilon_0,1)\to Q$ by letting $c|_{(-\varepsilon_0,0)}=\varphi_x\circ \tilde{c}_0|_{(-\varepsilon_0,0)}$. Clearly $c(-\varepsilon_0,0)$ belongs to a component of $\Sigma_1$ and the isomorphism class of the isotropy groups along this component is the same as the one along $S$. Moreover, since $c(-\varepsilon_0,\varepsilon_0)$ is the projection of the smooth geodesic $\tilde{c}_0(-\varepsilon_0,\varepsilon_0)$ in $\widetilde{U}_x$, the prolongation $c(-\varepsilon_0,1)$ of $S$ is the underlying path of a smooth orbifold geodesic passing through $x\in \Sigma_0$. 

Similarly, at $c(1)\in \Sigma_0$, we have one of the two possible situations: either $S$ ends at $c(1)$, or $S$ can be prolongated to the right to a continuous path $c(0,1+\varepsilon_1)$ with $c(1,1+\varepsilon_1)\subset \Sigma_1$, for some $\varepsilon_1>0$. 
\end{proof}

\begin{ex}\label{eg_closed_geodesic_1}
 Suppose the singular locus of a $3$-orbifold $\mathcal{Q}$ is as in Figure \ref{locus}, where $\Sigma_1$ consists of three open connected components $S, S'\mbox{ and } S''$, and $\Sigma_0=\{x,y\}$. 
\begin{figure}[h]
\begin{center}
\leavevmode\hbox{}
\includegraphics[width=5.5cm]{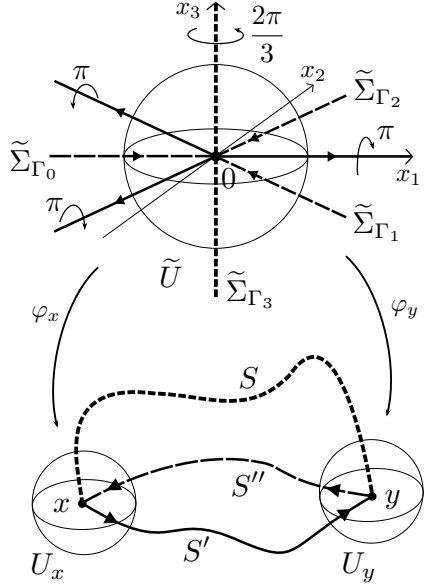}
\caption[Closed geodesics in Example \ref{eg_closed_geodesic_1}]{$S\cup\{x,y\}\simeq [0,1]$ and $S'\cup S''\cup\{x,y\}\simeq\mathbb{S}^1$ are totally geodesic 1-orbifolds contained in $\cl(\Sigma_1)$.}
\label{locus}
\end{center}
\end{figure}
Suppose that a fundamental neighbourhood $U_x$ at $x$ is uniformized by an orbifold chart $(\R^3,\Gamma,\varphi_x)$, where  $\Gamma$ is the subgroup of $SO(3)$ generated by $\gamma$, the rotation of angle $2\pi/3$ around the axis $0x_3$, and $\delta$, the rotation of angle $\pi$ around the axis $0x_1$. The group $\Gamma=\langle \gamma,\delta \mid \gamma^3=\delta^2=(\gamma\delta)^2=1\rangle$ is isomorphic to $D_3$, the dihedral group with six elements, and has four proper subgroups: one subgroup of order three $\Gamma_3=\langle \gamma\rangle$, and three subgroups of order two $\Gamma_0=\langle \delta\rangle$, $\Gamma_1= \langle \gamma\delta\rangle$ and $\Gamma_2= \langle \gamma^{-1}\delta\rangle$, which are conjugate to each other: $\Gamma_0=\gamma\Gamma_1\gamma^{-1}=\gamma^{-1}\Gamma_2\gamma$.

For each $i=0,1,2$ and 3, let $\widetilde{\Sigma}_{\Gamma_i}$ be the one dimensional subspace fixed by the subgroup $\Gamma_i$. Since each element of order two in $\Gamma$ leaves the subspace $\widetilde{\Sigma}_{\Gamma_3}$ invariant and maps any vector $v\in \widetilde{\Sigma}_{\Gamma_3}$ to $-v$, the subspace $\widetilde{\Sigma}_{\Gamma_3}$ projects via $\varphi_x$ to $\{x\}\cup(S\cap U_x)$ and the component $S$ {\it terminates} at $x$, or equivalently, $x$ is an \emph{end} of $S$. On the other hand, $\widetilde{\Sigma}_{\Gamma_0}=\gamma^{-1}\widetilde{\Sigma}_{\Gamma_1} = \gamma\widetilde{\Sigma}_{\Gamma_2}$, which implies that $\varphi_x(\widetilde{\Sigma}_{\Gamma_0})=\varphi_x(\widetilde{\Sigma}_{\Gamma_1})=\varphi_x(\widetilde{\Sigma}_{\Gamma_2})=((S'\cup S'')\cap U_x)\cup\{x\}$. In this case, the point $x\in \Sigma_0$ is not an end of either $S'$ or $S''$, and the component $S'$ \emph{extends} $S''$ past $x$ and vice-versa. 

Suppose that a fundamental neighbourhood $U_y$ at $y$ is uniformized by an orbifold chart $(\R^3,\Gamma, \varphi_y)$ such that, as before, $y$ is an end of $S$ and the components $S'$ and $S''$ are each other's extension through $y$. 

The closure $\cl(S)=S\cup\{x,y\}\simeq[0,1]$ has the structure of a compact one-dimensional non-effective orbifold $\mathcal{S}$ such that, in the associated reduced orbifold structure $\mathcal{S}_{\rm eff}$, any point in $S$ is a regular point and $x,y\in\fr(S)$ are singular points having isotropy $\Gamma_x\simeq\Gamma_y\simeq\Gamma/\Gamma_3\simeq\Z_2$. Hence $\mathcal{S}_{\rm eff}$ is diffeomorphic to the orbifold $\mathcal{I}$, introduced before Lemma \ref{lem:frontier_0}. An example of closed geodesic contained in $\cl(S)$ is given by a constant speed parametrization of $\cl(S)$ that starts at $x$, goes along $S$ toward $y$, is reflected by the element of order two in $\Gamma_y$, and travels back along $S$ toward $x$ where is reflected by the element of order two in $\Gamma_x$. 

We distinguish another class of closed geodesics in the closure $\cl(\Sigma_1)$: the union $S'\cup S''\cup\{x,y\}$ is a totally geodesic embedded $\mathbb{S}^1$, and any constant speed parametrization of a closed path going around $S' \cup S''\cup \{x,y\}$ gives rise to a closed geodesic of positive length in  $\mathcal{Q}$. 
\end{ex}

\begin{figure}[h]
\begin{center}
\leavevmode\hbox{}
\includegraphics[width=7.5cm]{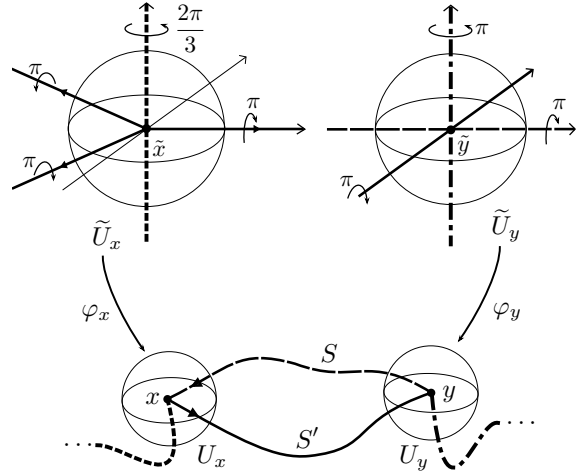}
\caption[Closed geodesics in Example \ref{eg_closed_geodesic_12}]{  $S\cup S'\cup\{x,y\}\simeq[0,1]$ is a totally geodesic 1-orbifold in $\cl(\Sigma_1)$.}
\label{locus3}
\end{center}
\end{figure}

\begin{ex}\label{eg_closed_geodesic_12} In Figure \ref{locus3}, assume that $x$ is a singular point in a $3$-orbifold $\mathcal{Q}$ as in Example \ref{eg_closed_geodesic_1}, and that the isotropy group at $y$ is the Klein group of four elements, $\Z_2\times\Z_2$, which acts on $\R^3$ by rotations of angle $\pi$ around the three orthogonal axes. As in the previous example, $x\in \fr(S)\cap\fr(S')$ is not an end for either $S$ or $S'$ and the components $S$ and $S'$ extend each other through $x$. However, this time, the point $y$ is an end for both $S$ and $S'$. The union $S\cup S'\cup \{x,y\}$ has the structure of a $1$-orbifold with underlying space homeomorphic to $[0,1]$, and, as before, $S\cup S'\cup \{x,y\}$ gives rise to a closed geodesic of positive length in $\mathcal{Q}$.
\end{ex}

In general, the frontier of an open connected component $S\subseteq\Sigma_1$ consist of either one point or two distinct points in $\Sigma_0$, and examples of the latter case are presented in \ref{eg_closed_geodesic_1} and \ref{eg_closed_geodesic_12}. We next consider the case when the frontier $\fr(S)$ consist of a single point $x\in\Sigma_0$.

Using the notation in the proof of Lemma \ref{lem:frontier_0}, we assume that $c(1)=c(0)=x$ and let $\tilde{c}_0\colon [0,\varepsilon_0)\to \widetilde{U}_x$ and $\tilde{c}_1\colon (1-\varepsilon_1 ,1]\to \widetilde{U}_x$ be geodesic paths in $\widetilde{U}_x$ such that  $\varphi_x\circ\tilde{c}_0=c|_{[0,\varepsilon_0)}$ and $\varphi_x\circ\tilde{c}_1=c|_{(1-\varepsilon_1,1]}$. Assume $\widetilde{U}_x$ is small enough so that $0<\varepsilon_0<1-\varepsilon_1<1$.   Let $\Gamma_0$ and $\Gamma_1$ be the subgroups of $\Gamma_x$ that fix the vector subspaces $V_0$ and $V_1$ in $T_{\tilde{x}}\widetilde{U}_x$ spanned by $\v_0=\dot{\tilde{c}}_0(0)$ and $\v_1=\dot{\tilde{c}}_1(1)$, respectively.  Although the groups $\Gamma_0$ and $\Gamma_1$ are isomorphic, they need not be conjugate in $\Gamma_x$. 

If $\Gamma_0$ and $\Gamma_1$ are not conjugate in $\Gamma_x$, then, as in Lemma \ref{lem:frontier_0}, the component $S$ either terminates at $x$ from both ends; or $S$ can be prolongated past $x$ to different components of $\Sigma_1$ at either one end or both ends. 

Assume now that $\Gamma_0$ and $\Gamma_1$ are conjugate in $\Gamma_x$. If $\gamma\in\Gamma_x$ is such that $\Gamma_0=\gamma\Gamma_1\gamma^{-1}$, then $\gamma.V_1=V_0$ and either $\gamma\v_1=\v_0$ or $\gamma\v_1=-\v_0$. We show that the latter is not possible. Indeed, if we suppose that $\gamma\v_1=-\v_0$, then $\gamma\tilde{c}_1(1-t)=\tilde{c}_0(t)$ and $\gamma\dot{\tilde{c}}_1(1-t)=-\dot{\tilde{c}}_0(t)$ for all $0\le t\le\min(\varepsilon_0,\varepsilon_1)$. Composing with $\varphi_x$, this yields $$ c(1-t)=c(t) \mbox{ and } \dot{c}(1-t)=-\dot{c}(t) \mbox{ for all } 0\le t\le\min(\varepsilon_0,\varepsilon_1).$$ Covering $S$ with fundamental neighbourhoods at points $c(t)$ and using the corresponding orbifold change of charts, the above condition implies that $c(1-t)=c(t)$ and $\dot{c}(1-t)=-\dot{c}(t)$ for all $t\in(0,1)$. But then $\dot{c}(\frac{1}{2})=-\dot{c}(\frac{1}{2})$, which contradicts the fact that $c\colon (0,1)\to S$ is a non-zero constant speed parametrization for $S\subseteq\Sigma_1$. 

Thus, if $\Gamma_0=\gamma\Gamma_1\gamma^{-1}$, then $\gamma\v_1=\v_0$ and this implies that the component $S$ extends past $x$ into itself. In other words, the closure $\cl(S)\simeq\mathbb{S}^1$ has the structure of a totally geodesic embedded circle, and, in this case, $c\colon [0,1]\to Q$ is the underlying path of a closed geodesic in $\mathcal{Q}$ passing through $x$. 

In conclusion, constant speed parametrizations of an open connected component $S\subseteq\Sigma_1$ can be extended smoothly past the frontier points of $S$, and these extended paths are underlying paths of smooth orbifold geodesics in $\mathcal{Q}$.  In our next result we use these extended geodesics to show the existence of closed geodesics  when $\mathcal{Q}$ is compact, these extended paths also give rise to closed geodesics of positive length in $\mathcal{Q}$.
 
 \begin{theorem}\label{thm:singular_1}
Suppose $\mathcal{Q}$ is a compact effective Riemannian orbifold such that $\Sigma_1\neq\varnothing$. Then there exists at least one closed geodesic of positive length in $\mathcal{Q}$. 
\end{theorem}

\begin{proof}
Let $S\neq\varnothing$ be a connected component of $\Sigma_1$. By Proposition \ref{prop:strata}, $S$ has the structure of a $1$-dimensional manifold without boundary which is totally geodesic in $\mathcal{Q}$. If $S$ is closed, then a constant speed parametrization of $S$ gives a closed geodesic on $\mathcal{Q}$, and this geodesic is contained entirely in $\Sigma_1\subseteq\cl(\Sigma_1)$.

Assume $S$ is open. By Proposition \ref{prop:finitely_many_components} there are only finitely many components in $\Sigma_1$, and by Lemma \ref{lem:frontier_0}, $S$ can be prolongated within $\cl(\Sigma_1)$ to a continuous path $c\colon [a,b]\to\Sigma_1\cup\Sigma_0$ with $a\le0$ and $1\le b$, such that $c([a,b])$ is maximal in the following sense. Either the points $c(a)$ and $c(b)$ are end points for the components $c(a,a+\varepsilon)$ and $c(b-\varepsilon,b)$, for some $\varepsilon>0$; or $c(a)=c(b)$  are not end points, and the component $c(a,a+\varepsilon)$ prolongates to the left into the component $c(b-\varepsilon,b)$. In the latter case, it is clear that $c([a,b])\simeq\mathbb{S}^1$ is the underlying path of a closed geodesic in $\mathcal{Q}$. Assume the former and note that $c(a)$ and $c(b)$ need not be distinct in $Q$ (see Example \ref{eg_closed_geodesic_12}). 
Let $\tilde{c}_a\colon [a,a+\varepsilon)\to \widetilde{U}_{c(a)}$ and $\tilde{c}_b\colon (b-\varepsilon,b]\to\widetilde{U}_{c(b)}$ be local lifts of $c$ to fundamental uniformizing charts at $c(a)$ and $c(b)$, for some $\varepsilon>0$. Because $c(a)$ and $c(b)$ are end points, there exist $\gamma_a\in\Gamma_{c(a)}$ and $\gamma_b\in\Gamma_{c(b)}$ such that $\gamma_a\dot{\tilde{c}}_a(a)=-\dot{\tilde{c}}_a(a)$ and $\gamma_b\dot{\tilde{c}}_b(b)=-\dot{\tilde{c}}_b(b)$. Let $a<t_1<a+\varepsilon<t_2<\ldots<t_{k-1}<b-\varepsilon<t_k<b$ be a subdivision of the interval $[a,b]$, and let $$(\tilde{c}_a|_{[a,t_1]},h_1,\tilde{c}_1,h_2,\ldots,\tilde{c}_{k-1},h_k,\tilde{c}_b|_{[t_k,b]},\gamma_b)$$ be a representative of an orbifold geodesic with underlying path $c([a,b])$, where $\tilde{c}_i\colon [t_i,t_{i+1}]\to\widetilde{U}_{c(t_i)}$ are local lifts of $c|_{[t_i,t_{i+1}]}$ and $h_i$ are change of charts such that $h_{i}(\tilde{c}_{i-1}(t_{i}))=\tilde{c}_{i}(t_i)$ and $h_{i}(\dot{\tilde{c}}_{i-1}(t_{i}))=\dot{\tilde{c}}_{i}(t_i)$. Let $\tilde{c}_i^-$ denote the paths $\tilde{c}_i$ with the reversed orientation. Then $$(\tilde{c}_a|_{[a,t_1]},h_1,\ldots,\tilde{c}_{k-1},h_k,\tilde{c}_b|_{[t_k,b]},\gamma_b,\tilde{c}^-_b|_{[t_k,b]},h_k^{-1},\tilde{c}^-_{k-1},\ldots,h_1^{-1},\tilde{c}^-_a|_{[a,t_1]},\gamma_a)$$ represents a closed orbifold geodesic whose length is twice the length of $c([a,b])$.
\end{proof}

\begin{rem}\label{rem:reduction} Theorems \ref{thm:closed_stratum} and \ref{thm:singular_1} allow us to reduce the problem of existence of closed geodesics of positive length on compact orbifolds to the case of compact orbifolds with only zero dimensional singular locus. \end{rem}

\section{Developable orbifolds}\label{sec:develop_geodesics}
Let $\mathcal{Q}$ be an effective $n$-dimensional compact connected developable Riemannian orbifold. We write $\mathcal{Q}$ as the orbifold quotient $M/\Gamma$, where $M$ is the orbifold universal covering of $\mathcal{Q}$ and $\Gamma=\pi_1^{orb}(\mathcal{Q})$ is the orbifold fundamental group. That is, $M$ is a connected, simply connected complete smooth Riemannian manifold (with the natural Riemannian structure pulled back from $\mathcal{Q}$) and $\Gamma$ is a discrete subgroup of the group $\mbox{Isom}(M)$ acting properly  and cocompactly by isometries on $M$. In short, we will say that $\Gamma$ acts \emph{geometrically} on $M$. Let $Q$ be the underlying topological space of $\mathcal{Q}=M/\Gamma$ and let $\pi\colon M\to Q$ be the natural projection map. 

Recall from \S\ \ref{ssec:paths} that closed geodesics of positive length on $\mathcal{Q}$ are in one-to-one correspondence with equivalence classes of pairs $(\tilde{c},\gamma)$, where $\tilde{c}\colon [0,1]\to M$ is a non-constant geodesic segment in $M$ and $\gamma\in\Gamma$ is an isometry of $M$ such that $\gamma \tilde{c}(1)=\tilde{c}(0) \mbox{ and } \gamma \dot{\tilde{c}}(1)=\dot{\tilde{c}}(0).$ Two pairs $(\tilde{c},\gamma)$ and $(\tilde{c}',\gamma')$ are equivalent if and only if there is an isometry $\delta \in \Gamma$ such that $\tilde{c}' = \delta.\tilde{c}$ and $\gamma' = \delta\gamma\delta^{-1}$.

As shown in the previous sections, the problem of existence of closed geodesics on compact orbifolds is particularly interesting in the case when the singular locus of the orbifold has dimension zero (cf. Remark \ref{rem:reduction}).  In this section, we show that if such an orbifold is developable and has odd dimension, then it admits at least one closed geodesics of positive length. This is a consequence of the following more general result. 

\begin{theorem}\label{thm:order2point}
Let $\mathcal{Q}$ be a compact connected effective developable Riemannian orbifold. Suppose $\mathcal{Q}$ has an isolated singular point whose isotropy group has even order. Then $\mathcal{Q}$ admits a closed geodesic of positive length.
\end{theorem}
\begin{proof}

Suppose $\mathcal{Q}=M/\Gamma$ with $M$ simply connected Riemannian manifold and $\Gamma\simeq\pi_1^{orb}(\mathcal{Q})$. Let $x\in M$ be an isolated singular point such that the isotropy group $\Gamma_x$ has even order. Let $\gamma\in\Gamma_x$ be an element of order two. Since $\gamma$ acts as a linear  isometry on $T_xM$ fixing only the origin, $\gamma$ is the inversion in the origin, i.e. $\gamma\v=-\v$ for all $\v\in T_xM$.

If $\Gamma=\Gamma_x$, then $\Gamma$ is finite and $M$ is necessarily compact. Any closed geodesic in $M$ gives a closed orbifold geodesic in $\mathcal{Q}$. 

Assume now that $\Gamma\neq\Gamma_x$ and let $\delta\in\Gamma\smallsetminus\Gamma_x$. Then $\delta x\neq x$ and let $\tilde{c}\colon [0,1]\to M$ be a geodesic segment connecting $x$ to $\delta x$. Let $\tilde{c}^-$ denote the geodesic $\tilde{c}$ with the reversed orientation, i.e. $\tilde{c}^-(t)=\tilde{c}(1-t), t\in[0,1]$. Let $\gamma \tilde{c}$ be the translate of $\tilde{c}$ by the isometry $\gamma$. Then $\gamma \tilde{c}$ is a geodesic connecting $x$ to $\gamma\delta x$. Since $x$ is an isolated singular point, $\gamma$ does not fix $\tilde{c}$ and $$\gamma\dot{\tilde{c}}(0)=-\dot{\tilde{c}}(0)=\dot{\tilde{c}}^-(1).$$ Let $\tilde{c}'\colon [0,1]\to M$ be a constant speed reparametrization of the concatenation $\tilde{c}^-*\gamma \tilde{c}$. Then $\tilde{c}'$ is a smooth geodesic segment from $\tilde{c}'(0)=\delta x$ to $\tilde{c}'(1)=\gamma\delta x$ satisfying $\dot{\tilde{c}}'(0)=\dot{\tilde{c}}^-(0)=-\dot{\tilde{c}}(1)$ and $\dot{\tilde{c}}'(1)=\gamma\dot{\tilde{c}}(1)$. We distinguish two possible situations depending on whether $\gamma$ fixes $\delta x$ or not (Figure \ref{fig:even_geodesics}).

\begin{figure}[h]
\begin{center}
\leavevmode\hbox{}
\includegraphics[width=11cm]{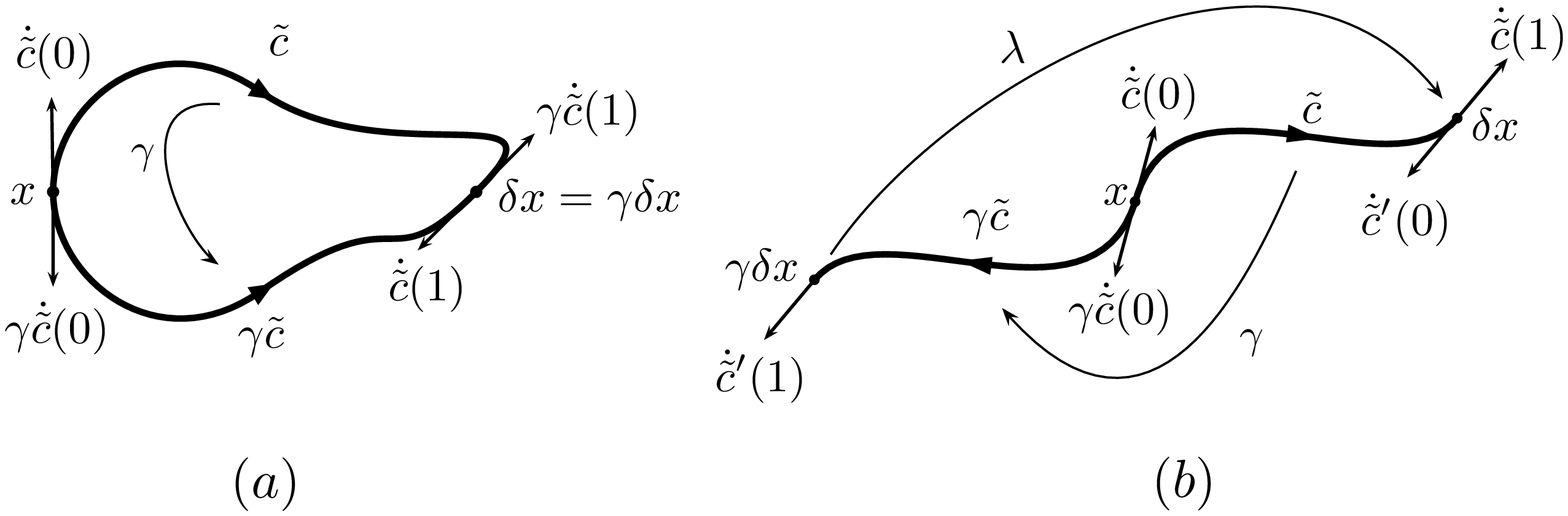}
\caption{Proof of Theorem \ref{thm:order2point}.}
\label{fig:even_geodesics}
\end{center}
\end{figure}

If $\gamma\delta  x=\delta  x$, then $\tilde{c}'(0)=\tilde{c}'(1) = \delta x$ (Figure \ref{fig:even_geodesics} $(a)$). Because $\delta x$ has the same orbit type as $x$, $\delta x$ is an isolated singular point and because $\gamma\in \Gamma_{\delta x}$ has order two, $$\dot{\tilde{c}}'(1)=\gamma\dot{\tilde{c}}(1)=-\dot{\tilde{c}}(1)=\dot{\tilde{c}}^-(0)=\dot{\tilde{c}}(0).$$ This shows that $\tilde{c}'$ is a smooth closed geodesic in $M$, and thus $(\tilde{c}',1)$ defines a closed orbifold geodesic in $\mathcal{Q}$.

Assume now that $\gamma\delta x\neq\delta x$ (Figure \ref{fig:even_geodesics} $(b)$). Let $\lambda=\delta\gamma\delta^{-1}\gamma$. Then $$\lambda\tilde{c}'(1) = \lambda(\gamma\delta x) = \delta\gamma\delta^{-1}\gamma\gamma\delta x = \delta x=\tilde{c}'(0), \mbox{ and }$$  $$\lambda(\dot{\tilde{c}}'(1)) = \delta\gamma\delta^{-1}\gamma\gamma\dot{\tilde{c}}(1) = \delta\gamma\delta^{-1}\dot{\tilde{c}}(1) = -\dot{\tilde{c}}(1) = \dot{\tilde{c}}'(0),$$ since $\delta\gamma\delta^{-1}\in\Gamma_{\delta x}$ and has order two. Thus, the pair $(\tilde{c}',\lambda)$ represents a closed geodesic of positive length in $\mathcal{Q}$. 
\end{proof}

\begin{cor}\label{cor:odd_zero}
Suppose $\mathcal{Q}$ is an odd-dimensional compact developable orbifold which has an isolated singular point. Then $\mathcal{Q}$ admits at least one closed geodesic of positive length. In particular, this is true if the singular locus is zero-dimensional. 
\end{cor}
\begin{proof}
Assume $\mathcal{Q}=M/\Gamma$ and let $\tilde{x}\in M$ be a lift of an isolated singular point in $Q$.  The isotropy group $\Gamma_{\tilde{x}}$ of $\tilde{x}$ acts freely and orthogonally on the unit sphere in the tangent space $T_{\tilde{x}}M$, and this sphere has even dimension, since $M$ is odd-dimensional. Thus $\Gamma_{\tilde{x}}\simeq \Z_2$ is cyclic of order two, and the conclusion follows from Theorem \ref{thm:order2point}.
\end{proof}

In the light of Remark \ref{rem:reduction}, the above corollary shows that in the developable case, the problem of existence of closed geodesics can be reduced to orbifolds of even dimension and with zero-dimensional singular locus. Indeed, if $\mathcal{Q}$ is a compact developable $n$-orbifold and $k>0$ is the smallest singular dimension such that $\Sigma_k$ is not empty, then $\mathcal{Q}$ admits closed geodesics of positive length by Theorem \ref{thm:singular_1} when $k=1$, and by Theorem \ref{thm:closed_stratum} and Corollary \ref{cor:odd_zero} when $k>1$ and $k$ is odd.  If one could apply Theorem \ref{thm:closed_stratum} when $k$ is even, then the existence of closed geodesics would follow for all compact developable orbifolds. However, we note that for even-dimensional orbifolds, the isotropy group of an isolated singular point admits a free orthogonal action on a sphere which has odd dimension, and because, in this case, the order of the isotropy group need not be even (cf. \cite[Theorem 6.1.11]{Wol:67Spa:aa}), Theorem \ref{thm:order2point} no longer applies. While we are unable to prove Corollary \ref{cor:odd_zero} for all even dimensions, our next theorem still shows existence of closed geodesics on $\mathcal{Q}$ in certain cases when $k$ is even.

\begin{theorem}\label{thm:odd_strata}
Let $\mathcal{Q}$ be a compact connected effective Riemannian developable $n$-orbifold, and let $k>0$ be the smallest singular dimension such that $\Sigma_k\neq\varnothing$. There exists at least one closed geodesic of positive length on $\mathcal{Q}$ if either $k$ is odd or $k\ge n/2$ and $n-k$ is odd.
\end{theorem}

\begin{proof} Let $\mathcal{Q}=M/\Gamma$, with $M$ simply connected Riemannian manifold and $\Gamma\subset{\rm Isom}(M)$ acting geometrically on $M$ with orbifold quotient $\mathcal{Q}$. Let $S\subseteq\Sigma_k$ be a connected component of the singular stratum of smallest positive dimension $k$. If $S$ is closed, the conclusion follows from Theorem \ref{thm:closed_component}. For the rest of the proof, we will assume that $S$ is not closed.

If $k=1$, the existence of closed geodesics on $\mathcal{Q}$ is given by Theorem \ref{thm:singular_1}. If $k>1$, then by Theorem \ref{thm:closed_stratum} the closure ${\rm cl}(S)$ has the structure of a compact orbifold $\mathcal{S}$ which is totally geodesic in $\mathcal{Q}$, and the associated effective orbifold $\mathcal{S}_{\rm eff}$ has only zero-dimensional singular locus or is a smooth manifold. In the latter case, by Lusternik and Fet \cite{Lyu:51Var:aa}, $\mathcal{S}_{\rm eff}$ admits closed geodesics of positive length and these geodesics give rise to closed orbifold geodesics in $\mathcal{Q}$. Assume now that $\mathcal{S}_{\rm eff}$ is a nontrivial orbifold. Since $\mathcal{Q}$ is developable, the orbifold $\mathcal{S}_{\rm eff}$ is developable and $\mathcal{S}_{\rm eff}=N/\Gamma'$, where $N$ is a totally geodesic $k$-submanifold of $M$. 

Let $x\in {\rm fr}(S)$ be a singular point in $\mathcal{S}_{\rm eff}$ and let $\tilde{x}\in N$ be a lift of $x$. Let $\Gamma'_x$ be the isotropy group of $x$ in $\mathcal{S}_{\rm eff}$. Since $x$ is an isolated singular point in $\mathcal{S}_{\rm eff}$, the group $\Gamma'_x$ acts linearly by isometries on the tangent space $T_{\tilde{x}}N$ fixing only the origin.  Thus $\Gamma'_x$ acts freely and orthogonally on the unit sphere $\mathbb{S}^{k-1}$ in $T_{\tilde{x}}N$. 

If $k$ is odd, then $\Gamma'_x\simeq\Z_2$ is cyclic of order two, and the existence of a closed geodesic of positive length on $\mathcal{S}_{\rm eff}$, and hence on $\mathcal{Q}$, follows from Theorem \ref{thm:order2point}.

Assume now that $k\ge n/2$ and $n-k$ is odd. Let $(\widetilde{U}_x,\Gamma_x,\varphi_x)$ be a fundamental $n$-orbifold chart at $x$ (i.e. with respect to the orbifold structure $\mathcal{Q}$). If $k>n/2$, then there is a unique connected component $\widetilde{S}\subset\widetilde{U}_x$ that projects via $\varphi_x$ to $S\cap \varphi_x(\widetilde{U}_x)$. Indeed,  if $\widetilde{S}'$ is a different such component, then the intersection $\widetilde{S}'\cap\widetilde{S}$ is nonempty (since it contains $x$) and, because both components have dimension $k>n/2$, $\widetilde{S}'\cap\widetilde{S}$ gives rise to a singular component in $\mathcal{Q}$ of positive dimension which is less than $k$. This contradicts the fact that $k$ is the smallest positive singular dimension in $\mathcal{Q}$, and proves that the component $\widetilde{S}$ is unique. Let $\Gamma_S$ be the subgroup of $\Gamma_x$ such that $\widetilde{S}\cup\tilde{x}={\rm Fix}(\Gamma_S)$. Then $\Gamma_S$ is normal in $\Gamma_x$ and $\Gamma'_x\simeq \Gamma_x/\Gamma_S$. Note that if $\gamma\in \Gamma_x\smallsetminus\Gamma_S$, then $\gamma$ acts freely and orthogonally on the unit $(n-k-1)$-sphere in the orthogonal complement $T_{\tilde{x}}\widetilde{S}^\perp$. Since $n-k-1$ is even, $\gamma$ has order two and thus $\Gamma_x'$ is cyclic of order two. Finally, if $k=n/2$ and $n-k$ is odd, then $k$ is odd and, as in the first part of the proof, the group $\Gamma_x'$ is cyclic of order two. Again, by Theorem \ref{thm:order2point}, there exists a closed geodesic of positive length in $\mathcal{S}_{\rm eff}$, and therefore in $\mathcal{Q}$. 
\end{proof}

As an application, we investigate the existence of closed geodesics in low dimensional orbifolds. We first note that if $\mathcal{Q}$ is a two-dimensional compact developable orbifold, then $\mathcal{Q}$ is finitely covered by a compact manifold $M$ (cf. \cite[Theorem 2.5]{Sco:83The:aa}) and, because any closed geodesic in $M$ gives rise to a closed geodesic in $\mathcal{Q}$, the orbifold $\mathcal{Q}$ admits closed geodesics of positive length. Together with the result of \cite{Gur:06Clo:aa} for non-developable orbifolds, this proves the following.

\begin{prop}\label{prop:dimension2}
All compact $2$-dimensional orbifolds admit at least one closed geodesic of positive length.
\end{prop}

We next use Theorem \ref{thm:odd_strata} to prove the existence of closed geodesics on compact developable orbifolds in dimensions 3, 5 and 7. Again, because of the result of \cite{Gur:06Clo:aa} for non-developable orbifolds, we have the following. 
 
\begin{cor}\label{cor:dimension_3_5}
If $\mathcal{Q}$ is a compact orbifold with $\dim(\mathcal{Q})$ equal to 3, 5 or 7, then $\mathcal{Q}$ admits a closed geodesic of positive length. 
\end{cor}

\begin{proof} 
By Lyusternik and Fet \cite{Lyu:51Var:aa}, we can assume that the singular locus is nonempty, so $\Sigma\neq\varnothing$, and by part (a) of Theorem 5.1.1 of Guruprasad and Haefliger \cite{Gur:06Clo:aa}, we can assume that $\mathcal{Q}$ is developable. Let $Q=\bigsqcup\limits_{\ell=0}^n\Sigma_\ell$ be the natural stratification by singular dimension. By Corollary \ref{cor:no_zero_singularity}, we can assume that $\Sigma_0\neq\varnothing$. Let $k>0$ denote the smallest singular dimension such that $\Sigma_k\neq\varnothing$

If $\dim(\mathcal{Q})=3$ and $k=1,2$ or $3$, the conclusion follows from Theorem \ref{thm:odd_strata}. 

If $\dim(\mathcal{Q})=5$ or $7$ and $k=2$ the conclusion follows from Proposition \ref{prop:dimension2}. For all the other possible values $k$ the conclusion follows from Theorem \ref{thm:odd_strata}. \end{proof}

Finally, recall the following existence result for compact developable orbifolds (cf.   \cite[Theorem 5.1.1]{Gur:06Clo:aa}, \cite[Proposition 2.16]{Ale:11On-:aa} or \cite[Theorem 2]{Dra:14Clo:aa}).

\begin{theorem}\label{thm:existence1}
A developable compact connected Riemannian orbifold $\mathcal{Q}$ has a closed geodesic of positive length if the orbifold fundamental group $\pi_1^{orb}(\mathcal{Q})$ is finite or if it contains a hyperbolic element.
\end{theorem}

The only situation not covered by Theorem \ref{thm:existence1} is when $\Gamma$ is infinite and each of its elements is an elliptic isometry. 
 Since elliptic isometries have finite order, $\Gamma$ is an {\it infinite torsion group}. Moreover, since the action is cocompact, $\Gamma$ is finitely presented and has finitely many conjugacy classes of isotropy groups. The latter implies that $\Gamma$ has finite exponent. While examples of infinite torsion groups that are finitely generated and even of finite exponent are known to exist, there are no examples known to be finitely presentable (as also noted in \cite[Remark 5.1.2]{Gur:06Clo:aa}).
 
\begin{rem}\label{rem:reduction_good} 
In the light of Theorem \ref{thm:existence1}, our approach reduces the problem of existence of closed geodesics to the case of even dimensional compact developable orbifolds with zero-dimensional singular locus and orbifold fundamental group infinite torsion and of finite odd exponent. 
\end{rem}
In \cite[Chapter 5]{Dra:11Clo:aa}, we present various group-theoretic properties that one can assume about the orbifold fundamental group of orbifolds like in Remark \ref{rem:reduction_good}. The simplest such case  would be in dimension four. In this case, each isotropy group admits a free action on the three sphere and thus it must be cyclic (of odd order).

\bibliography{/Users/George/Dropbox/BibReader/references}
\bibliographystyle{amsplain} 

\end{document}